\documentclass[ ]{article}

\usepackage[margin=1in]{geometry}

\usepackage[latin1]{inputenc}
\usepackage[english]{babel}
\usepackage[babel]{csquotes}
\usepackage[colorlinks=true, pdfstartview=FitV, linkcolor=colorLink, citecolor=colorCite, urlcolor=colorLink, linktocpage=true]{hyperref}
\usepackage[noadjust]{cite}
\usepackage{amssymb, amsmath, amsthm}
\usepackage[toc,page]{appendix}
\usepackage{bbm}
\usepackage{verbatim}
\usepackage{enumerate}
\usepackage{enumitem}
\usepackage{tikz}
\usepackage{graphicx}
\usepackage{pgfplots}
\usepackage{tocloft}

\usepackage{etoolbox} 

\theoremstyle{plain}
\newtheorem{theorem}{Theorem}[section]
\newtheorem{corollary}[theorem]{Corollary}
\newtheorem{lemma}[theorem]{Lemma}
\newtheorem{proposition}[theorem]{Proposition}

\theoremstyle{definition}
\newtheorem{definition}[theorem]{Definition}
\newtheorem{remark}[theorem]{Remark}
\newtheorem*{acknowledgements}{Acknowledgements}

\numberwithin{equation}{section}

\makeatletter
\renewcommand{\paragraph}{%
\@startsection{paragraph}{4}%
{\z@}{2ex \@plus 1ex \@minus .2ex}{-1em}%
{\normalfont\normalsize\bfseries}%
}
\makeatother

\def\R{\mathbb{R}}
\def\P{\mathbb{P}}
\def\F{\mathcal{F}}
\def\B{\mathcal{B}}

\def\Z{\mathcal{Z}}
\def\H{\mathcal H}
\def\E{\mathbb{E}}
\def\eps{\varepsilon}

\def\EE{\mathrm{E}}
\def\UU{\mathrm{U}}
\def\BB{\mathrm{B}}

\renewcommand{\d}{{\mathrm d}}

\DeclareMathOperator\dom{dom}
\DeclareMathOperator\dist{dist}
\DeclareMathOperator\diam{diam}

\allowdisplaybreaks

\definecolor{colorLink}{RGB}{0,100,162}
\definecolor{colorCite}{RGB}{8,124,100}

\title{Multifractal analysis of Gaussian multiplicative chaos and applications}
\author{Federico Bertacco\footnote{Imperial College London, United Kingdom. E-mail: \textrm{f.bertacco20@imperial.ac.uk}}}
\date{ }
\begin{document}	

\maketitle

\begin{abstract}
Let $M_{\gamma}$ be a subcritical Gaussian multiplicative chaos measure associated with a general log-correlated Gaussian field defined on a bounded domain $D \subset \R^d$, $d \geq 1$. We find an explicit formula for its singularity spectrum by showing that $M_{\gamma}$ satisfies almost surely the multifractal formalism, i.e., we prove that its singularity spectrum is almost surely equal to the Legendre--Fenchel transform of its $L^{q}$-spectrum. Then, applying this result, we compute the lower singularity spectrum of the multifractal random walk and of the Liouville Brownian motion.
\end{abstract}

\section{Introduction}
\label{sec:intro}
The main goal of this paper is to prove that subcritical \emph{Gaussian Multiplicative Chaos} (GMC) measures associated with a large class of log-correlated Gaussian fields in all dimensions satisfy the multifractal formalism. This has been formally stated in \cite[Section~4]{RV_Review}, but to the best of our knowledge, it has never been proved rigorously. In order to achieve this result, we perform a careful analysis of the local mass concentration of GMC measures around thick points of the corresponding underlying field. Moreover, using this result, we provide an explicit expression for the lower singularity spectrum of the \emph{Multifractal Random Walk} (MRW) and of the \emph{Liouville Brownian Motion} (LBM). Before entering into the details of the main results, we briefly review the theory of GMC and of multifractal analysis of measures.

\paragraph{Gaussian multiplicative chaos}
Given a domain $D \subset \R^d$, $d \geq 1$, the theory of GMC, originally developed by Kahane \cite{Kahane}, aims to define rigorously random measures of the form 
\begin{equation}
\label{def:GMC}
M_{\gamma} (\d x) = e^{\gamma X(x)-\frac{1}{2} \gamma^2 \E[X(x)^2]} \d x \,,
\end{equation}
where $\d x$ denotes the Lebesgue measure, $\gamma$ is a real parameter, and $X$ is a log-correlated Gaussian field on $D$, i.e.\ a centred Gaussian field whose covariance kernel can be formally written as
\begin{equation*}
\E[X(x)X(y)] = -\log |x-y| + g(x,y)\,, \quad  x, y \in D \,,
\end{equation*}
where $g : D \times D \to \mathbb{R}$ is say bounded and continuous. Since the covariance kernel of $X$ has a logarithmic divergence along the diagonal, we cannot define the field $X$ pointwise. However, we can make rigorous sense of $X$ by viewing it as a random Schwartz distribution. Therefore, the definition of \eqref{def:GMC} is non-trivial as, a priori, we cannot exponentiate a random generalized function. In order to interpret \eqref{def:GMC} rigorously, we need to approximate $X$ via a regularizing procedure which involves a suitable collection of regularized random fields $(X_{\eps})_{\eps \in (0, 1]}$. The GMC measure associated with $X$ is then given by the limit of the sequence of approximating measures
\begin{equation*}
M_{\gamma}^{\eps} (\d x) = e^{\gamma X_{\eps}(x)-\frac{1}{2} \gamma^2 \E[X_{\eps}(x)^2]} \d x \,.
\end{equation*}
As long as the parameter $\gamma^2$ is strictly less than the critical value $2d$, which is usually called subcritical regime, it is well-known \cite{Berestycki_Elementary, Kahane, Revisited} that the sequence $(M_{\gamma}^{\eps})_{\eps \in (0, 1]}$ converges weakly in probability towards a non-degenerate measure $M_{\gamma}$. Moreover, it is known that $M_{\gamma}$ is almost surely non-atomic, but singular with respect to the Lebesgue measure. Many further properties of such measures concerning, among others, moments and multifractal behaviour are known. We refer to Section~\ref{sec:setup} for more details.

The original interest in defining GMC measures stemmed from the need of making rigorous Mandelbrot's model for energy dissipation in fully developed turbulence \cite{Mandelbrot1972}, but it has since been found applications in a wide range of fields: from mathematical finance \cite{RV_Finance} to mathematical physics \cite{DS_Inventiones}, but also random matrices \cite{Random_Matrices} as well as number theory \cite{Riemann_Zeta}. For a review on the theory of GMC and for further references and applications we refer to \cite{RV_Review}.

\paragraph{Multifractal analysis} 
The purpose of multifractal analysis is to finely describe the heterogeneity in distribution of measures whose mass is concentrated in a highly irregular way. For $\mu$ a non-negative finite measure supported on a domain $D \subset \R^d$, $d \geq 1$, we introduce the local H\"older exponent (or local dimension) of $\mu$ at $x \in D$ by letting
\begin{equation*}
\dim_{\mu}(x) :=  \lim_{r \searrow 0} \frac{\log \mu(B(x, r))}{\log r} \,,
\end{equation*}
provided the limit exists, where $B(x, r)$ denotes the closed ball centred at $x$ with radius $r$. Then the irregularity on the mass concentration of $\mu$ can be described via the dimension of the sets
\begin{equation*}
\EE_{\mu}(\alpha) := \left\{x \in D \, : \, \dim_{\mu}(x) = \alpha\right\} \,, \quad  \alpha \geq 0 \,.
\end{equation*}
In particular, we say that $\mu$ is a multifractal measure if the sets $\EE_{\mu}(\alpha)$ have positive Hausdorff dimensions for different values of $\alpha \geq 0$ belonging to an interval with non-empty interior. Hence, if $\mu$ is a multifractal measure, then the collection of sets $(\EE_{\mu}(\alpha))_{\alpha \geq 0}$ produces a decomposition of $D$ into a family of subfractals. 

The main objective of multifractal analysis is to compute the size of $\EE_{\mu}(\alpha)$, i.e.\ to find an expression for the \emph{singularity spectrum} of $\mu$, which is the function $\d_{\mu} : [0, \infty) \to [0, \infty)$ defined by
\begin{equation*}
\d_{\mu}(\alpha) := \dim_{\mathcal{H}}(\EE_{\mu}(\alpha)) \,, \quad \alpha \in [0, \infty) \,,
\end{equation*}
where $\dim_{\mathcal{H}}$ denotes the Hausdorff dimension. To this end, Frisch and Parisi \cite{Parisi} introduced the notion of \emph{multifractal formalism}, which is a heuristic principle used to establish an explicit connection between the singularity spectrum $\d_{\mu}$ and the $L^q$-spectrum of the measure $\mu$. We define the $L^q$-spectrum of $\mu$ by
\begin{equation*}
\tau_{\mu}(q) := \limsup_{r \searrow 0}\frac{\log \sup\left\{\sum_{i \in I} \mu(B(x_i, r))^q\right\}}{-\log r} \,, \quad q \in \R \,,
\end{equation*}
where $(B(x_i , r))_{i \in I}$ is a countable family of disjoint closed balls with radius $r$ centred at $x_i \in D$, and the supremum is taken over all such families. We say that $\mu$ satisfies the multifractal formalism if the following equality holds 
\begin{equation}
\label{eq_MultifractalFormalism}
\d_{\mu}(\alpha) = \tau_{\mu}^{*}(\alpha) \,, \quad  \forall \alpha \geq 0 \,.
\end{equation}
Here, $\tau_{\mu}^{*}$ refers to the Legendre--Fenchel transform of $\tau_{\mu}$ which is defined by setting
\begin{equation*}
\tau_{\mu}^{*}(\alpha) := \inf_{q \in \R}\left\{\alpha q + \tau_{\mu}(q)\right\} \wedge 0\,, \quad \alpha \geq 0\,.
\end{equation*}
Let us mention that a rigorous mathematical version of multifractal formalism was initially developed in \cite{Olsen} and we refer to it for further details.

In order to investigate the local regularity of functions, or in our case of paths of stochastic processes, we adopt a similar approach. More precisely, if $I \subset \R$ is an interval and $f: I \to \R^d$, $d \geq 1$, is a given function, we define the lower local H\"older exponent (or lower local dimension) of $f$ at $x \in I$ by setting
\begin{equation*}
\underline{\dim}_{f}(x) := \liminf_{r \searrow 0} \frac{\log |f(x+r) - f(x-r)|}{\log r}\,,
\end{equation*}
and we define the sets in which $f$ has lower local H\"older exponent $\alpha \geq 0$ as 
\begin{equation*}
\underline{\EE}_{f}(\alpha) := \left\{x \in I \, : \, \underline{\dim}_{f}(x)  = \alpha \right\} \,.
\end{equation*}
As in the case of measures, we are interested in finding an explicit expression for the lower singularity spectrum of $f$, which is the function $\underline{\d}_{f} : [0, \infty) \to [0, \infty)$ defined by
\begin{equation*}
\underline{\d}_{f}(\alpha) := \dim_{\H}(\underline{\EE}_{f}(\alpha)) \,,  \quad \alpha \in [0, \infty) \,.
\end{equation*}

\subsection{Main results}
We state here the main results of this paper, and we refer to Section~\ref{sec:main}, \ref{sec:MRW} and \ref{sec:LBM} for the precise statements and for further details.

\paragraph{Multifractal analysis of GMC}
For a bounded domain $D \subset \R^d$, $d \geq 1$, we consider a log-correlated Gaussian field on $D$, and we show that the associated subcritical GMC measure $M_{\gamma}$ satisfies the multifractal formalism, in the sense of \eqref{eq_MultifractalFormalism}. This result has been heuristically discussed in \cite[Section~4]{RV_Review}, but to the best of our knowledge it has never been proved rigorously. To be precise, in Theorem~\ref{th_MainTheorem} we show that for $\gamma^2 < 2d$ and $\alpha \geq 0$, it holds almost surely that
\begin{equation*}
\d_{M_{\gamma}}(\alpha) = \tau_{M_{\gamma}}^{*}(\alpha)
= \begin{cases}
d-\frac{1}{2}\left(\frac{d-\alpha}{\gamma} + \frac{\gamma}{2}\right)^2 \,, & \quad \text{ if } \alpha \in \left[\left(\sqrt{d} - \frac{|\gamma|}{\sqrt{2}}\right)^2,  \left(\sqrt{d} + \frac{|\gamma|}{\sqrt{2}}\right)^2 \right] \,, \\
0 \,, & \quad \text{ otherwise} \,,
\end{cases}
\end{equation*}
hence recovering precisely the prediction made in \cite{RV_Review}. Furthermore, as a by-product, we also obtain an explicit expression for the $L^q$-spectrum of $M_{\gamma}$.

Let us briefly flesh out the intuition behind the proof of our main result. The strategy of the proof is inspired by the multifractal analysis of multiplicative cascades which has been performed in \cite{Barral99}. The upper bound for $\d_{M_{\gamma}}$, i.e.\ the inequality $\d_{M_{\gamma}} \leq \tau_{M_{\gamma}}^{*}$, follows from the general theory of the multifractal analysis of measures. On the other hand, the lower bound is more involved to prove and it is based on the following heuristic. For $\gamma^2 < 2d$, it is known that the GMC measure $M_{\gamma}$ is carried by the set of $\gamma$-thick points of the underlying field $X$ which is defined by 
\begin{equation}
\label{eq_GammaThick}
\mathcal{T}_{\gamma} := \left\{x \in D \, : \, \lim_{\eps \searrow 0} \frac{X_{\eps}(x)}{-\log \eps} = \gamma\right\} \,,
\end{equation}
where $(X_{\eps})_{\eps \in (0, 1]}$ is a suitable regularization of $X$. The key to our proof is to study the local mass concentration of $M_{\gamma}$ around thick points of $X$. More precisely, following the idea of \cite[Theorem~4.1]{RV_Review}, we show that it exists a non-random exponent $\alpha_q \geq 0$ such that $M_{\gamma}$ has local dimension $\alpha_q$ at points in $\mathcal{T}_{q \gamma}$, for all $q^2 < 2d/\gamma^2$. Consequently, this implies that $M_{\gamma}$ has local H\"older exponent $\alpha_q$ on a set of full $M_{q \gamma}$-measure. This fact, together with some known properties of GMC measures, is enough to prove the lower bound for $\d_{M_{\gamma}}$, i.e.\ the inequality $\d_{M_{\gamma}} \geq \tau_{M_{\gamma}}^{*}$. 

\paragraph{Multifractal analysis of MRW and LBM}
The MRW has been first introduced in \cite{MRW} as a stochastic volatility model, and it can be simply defined as follows. Fix a time $T > 0$ and let $d \geq 1$, then the $d$-dimensional MRW $\Z^{\gamma}$ is defined for $\gamma^2 < 2$ as 
\begin{equation*}
\Z^{\gamma}_t := B_{M_{\gamma}([0, t])} \,, \quad t \in [0, T]\,,
\end{equation*}
where $M_{\gamma}$ is a GMC measure on $[0, T]$ and $B$ is an independent $d$-dimensional Brownian motion. In Theorem \ref{th_MainTheoremBMMT}, we find the relation between the lower singularity spectrum of the paths of $\Z^{\gamma}$ and the singularity spectrum of $M_{\gamma}$. More precisely, for $\gamma^2 < 2$ and $\alpha \geq 0$, we show that
\begin{equation*}
\underline{\d}_{\Z^{\gamma}}(\alpha) = \d_{M_{\gamma}}(2 \alpha) = \begin{cases}
d-\frac{1}{2}\left(\frac{1-2 \alpha}{\gamma} + \frac{\gamma}{2}\right)^2 \,, & \quad \text{ if } \alpha \in \left[\left(\frac{1}{\sqrt{2}} - \frac{|\gamma|}{2}\right)^2,  \left(\frac{1}{\sqrt{2}} + \frac{|\gamma|}{2}\right)^2 \right] \,, \\
0 \,, & \quad \text{ otherwise} \,,
\end{cases}
\end{equation*}
almost surely. 

The LBM has been simultaneously defined in \cite{Ber_LBM, RVG_LBM} as the canonical planar diffusion associated with the Liouville quantum gravity metric tensor. In this article, we will consider the $d$-dimensional LBM, for $d \geq 2$. More precisely, let $D \subset \R^d$ be a bounded domain, let $B$ be a $d$-dimensional Brownian motion started inside $D$, and let $X$ be an independent log-correlated Gaussian field on $D$. Then the LBM $\B^{\gamma}$ on $D$ can be formally defined for $\gamma^2 < 4$ as 
\begin{equation*}
\B^{\gamma}_t := B_{F_{\gamma}^{-1}(t)} \,, \qquad F_{\gamma}(t) := \int_0^{t \wedge T} e^{\gamma X(B_s) - \frac{1}{2} \gamma^2 \E[X(B_s)^2]} \d s \,, \quad t \geq 0 \,,
\end{equation*}
where $T$ is the first exit time of $B$ from $D$. As we will see in Section~\ref{sec:LBM}, adapting the original definition of planar LBM to higher dimensions is a straightforward task. Then, we focus on computing the lower singularity spectrum of the paths of $\B_{\gamma}$, and in Theorem~\ref{th_MainTheoremLBM}, we show that for $\gamma^2 < 4$ and $\alpha \geq 0$, it holds that
\begin{equation*}
\underline{\d}_{\B^{\gamma}}(\alpha) = \begin{cases}
2 \alpha-2\alpha\left(\frac{2\alpha-1}{2 \alpha \gamma} + \frac{1}{4} \gamma \right)^2 \,, & \quad \text{ if } \alpha \in \left[\left(\sqrt{2}+\frac{|\gamma|}{\sqrt{2}}\right)^{-2},  \left(\sqrt{2}-\frac{|\gamma|}{\sqrt{2}}\right)^{-2}\right] \,, \\
0 \,, & \quad \text{ otherwise} \,,
\end{cases}
\end{equation*}
almost surely. Along the proof of this result, we show that the measure $\mu_{\gamma}$ defined by $\mu_{\gamma}([s, t]) := F_{\gamma}(t) - F_{\gamma}(s)$, for $s \leq t \in [0, T]$, satisfies the multifractal formalism and we compute its singularity spectrum. 

\subsection{Structure of the paper} 
The reminder of this article is structured as follows. In Section~\ref{sec:setup}, we collect some definitions and results that are used in the rest of the article. More precisely, we start by recalling the definition of log-correlated Gaussian field, and we prove a lemma concerning the fluctuations of its convolution approximation. Then we recall some known properties of GMC measures, we collect some facts about the Hausdorff dimension, and finally we state some general results on multifractal analysis of measures. In Section~\ref{sec:main}, we state precisely our main result and we perform its proof. Sections~\ref{sec:MRW} and \ref{sec:LBM} are devoted to computing the lower singularity spectrum of the MRW and of the LBM, respectively. In Appendix~\ref{sec:ProofProp}, we prove Proposition~\ref{pr_RVCarrier}, which is the main step of the proof of our main result. Appendix~\ref{sec:finiteMoments} contains the proof of the finiteness of positive moments of the measure involved in the definition of the LBM. Finally, Appendix~\ref{sec:GaussianTool} collects some general results on Gaussian fields that are used throughout the paper.

\begin{acknowledgements}
The author would like to thank Prof.\ M.~Hairer for his constant support and guidance. We thank an anonymous referee for many helpful comments on an earlier version of this article. The author is very grateful to the Royal Society for financial support through Prof.\ M.~Hairer's research professorship grant RP\textbackslash R1\textbackslash 191065.
\end{acknowledgements}

\section{Preliminaries}
\label{sec:setup}
In this section, after introducing the basic notation, we collect some definitions and results on log-correlated Gaussian fields, on the theory of GMC, on the Hausdorff dimension, and on multifractal analysis of measures. 

\subsection*{Basic notation}
We let $\mathbb{N} := \{1, 2, 3, \dots\}$. For $d \geq 1$, we use $\R^d$ to indicate the $d$-dimensional Euclidean space. If $a$ and $b$ are two quantities, we use $a \lesssim b$ to denote the statement $a \leq C b$ for some constant $C > 0$ independent from the parameters of interest. 
Given a subset $D \subset \R^d$, we denote by $\bar{D}$ its closure and by $\partial D$ its boundary.

\subsection{Log-correlated Gaussian fields}
Given a bounded domain $D \subset \R^d$, $d \geq 1$, a log-correlated Gaussian field $X$ on $D$ is a Gaussian field whose covariance kernel takes the form
\begin{equation}
\label{eq_Covariance}
K(x, y) = -\log|x-y| + g(x, y) \,, \quad x, y \in D \,,
\end{equation}
where $g \in C(\bar{D} \times \bar{D})$. We adopt the convention to extend the covariance kernel $K$ to $\R^d \times \R^d$ by setting $K(x, y) = 0$ whenever $(x, y) \not \in D \times D$. Moreover, for \eqref{eq_Covariance} to be a covariance kernel, we need to require that it is symmetric and non-negative semi-definite. Since the covariance kernel $K$ has a singularity on the diagonal, the field $X$ does not make literal sense as a pointwise defined Gaussian field, but it can be rigorously interpreted as a random Schwartz distribution. Such a random generalized function can be characterized by the property that, for any compactly supported test function $\phi \in C_c^{\infty}(\R^d)$, the pairing $(X, \phi)$ produces a centred Gaussian random variable with variance
\begin{equation*}
\E[(X, \phi) (X, \phi)] = \int_{\R^d \times \R^d} \phi(x) K(x, y) \phi(y) \d x \d y \,.
\end{equation*}
The existence of such a stochastic process follows from a direct construction. Indeed, it can be easily verified that $K$, as defined in \eqref{eq_Covariance}, is the kernel of a self-adjoint Hilbert--Schmidt operator on $L^2(\R^d)$. In particular, such an operator is symmetric and compact, so by the spectral theorem there exist a non-increasing sequence $(\lambda_{n})_{n \in \mathbb{N}}$ of strictly positive eigenvalues and corresponding eigenfunctions $(\phi_n)_{n \in \mathbb{N}}$ that form an orthonormal basis of $\text{Ker}(K)^{\perp}$ in $L^2(\R^d)$. Then the field $X$ can be defined via its Karhunen--Lo\`{e}ve expansion
\begin{equation*}
X(x) = \sum_{n \in \mathbb{N}} Z_n \sqrt{\lambda_n} \phi_n \,,
\end{equation*}
where $(Z_n)_{n \in \mathbb{N}}$ is a collection of i.i.d.\ standard normal random variables defined on a common probability space $(\Omega, \F, \P)$. Note that the functions $\phi_n$ are supported on the domain $D$. It can be proved that the convergence of the above series takes place in the space $H^{- \alpha}(\R^d)$, for any $\alpha > 0$. We refer to \cite[Section~2]{Junnila_Imaginary} for more details.

In order to give a meaning to GMC measures, we need to define the exponential of a random Schwartz distribution. This is done through a regularization and limiting procedure. The most general and natural way to obtain a regularization of $X$ is through the so-called convolution approximation. If we let $\psi \in C^{\infty}_{c}(\R^d)$ be non-negative, radially symmetric, with compact support and unit mass, then the $\eps$-convolution approximation of the field $X$ is defined by 
\begin{equation}
\label{eq_ConvApprox}
X_{\eps}(x) : = X * \psi_{\eps}(x) = \int_{\R^d} X(x') \psi_{\eps}(x-x') \d x' \,, \quad  x \in \R^d \,, \quad  \eps \in (0, 1] \,,
\end{equation}
where $\psi_{\eps}(x) = \eps^{-d} \psi(\eps^{-1} x)$ for $x \in \R^d$. As we observed above, $X \in H^{-\alpha}(\R^d)$ for any $\alpha > 0$, and so the convolution $X * \psi_{\eps}$ is actually well-defined. Moreover, it is easy to check that the regularized field $X_{\eps}$ is a centred Gaussian field with covariance kernel given by
\begin{equation*}
K_{\eps}(x, y) = \int_{\R^d \times \R^d} \psi_{\eps}(x-x')K(x', y') \psi_{\eps}(y - y') \d x' \d y' \,, \quad x, y \in \R^d \,.
\end{equation*}
Furthermore, the following properties of the convolution approximation $(X_{\eps})_{\eps \in (0, 1]}$ are well known (see e.g.\ \cite[Lemma~2.8]{Junnila_Imaginary}):
\begin{enumerate}[label=(P\arabic*)]
\itemsep0em 
\item\label{pp_P1} for Lebesgue-almost every $(x, y) \in D \times D$,
\begin{equation*}
\lim_{\eps \searrow 0} K_{\eps}(x, y)= K(x,y) \,;
\end{equation*}
\item\label{pp_P3} there exists a finite constant $K > 0$ such that,
\begin{equation*}
\sup_{0 < \eps' < \eps \leq 1} \sup_{x, y \in D} \left| \E[X_{\eps}(x) X_{\eps'}(y)] + \log\left(|x-y| + \eps\right) \right| < K \,;
\end{equation*}
\item\label{pp_P4} for all $\eps \in (0, 1]$, the map $D \ni x \mapsto X_{\eps}(x)$ is almost surely continuous. 
\end{enumerate}

Let us mention that for particular types of log-correlated Gaussian fields there are some other natural approximations having properties \ref{pp_P1}, \ref{pp_P3}, and \ref{pp_P4}. For example, if $X$ is a two-dimensional \emph{Gaussian Free Field} (GFF), then one can use the regularization obtained through the circle average around each point in the domain (cf.\ \cite{DS_Inventiones}).

Before proceeding, we state and prove a lemma that allows to control the exponential moment of the fluctuations of the convolution approximation of log-correlated Gaussian fields. Such lemma is crucial in the proof of our main result, more specifically in the proof of Proposition~\ref{pr_RVCarrier}.
\begin{lemma}
\label{lm_supremumGaussian}
For $d \geq 1$, let $D \subset \R^d$ be a bounded domain. Consider a log-correlated Gaussian field $X$ on $D$ with covariance kernel \eqref{eq_Covariance} and let $(X_{\eps})_{\eps \in (0, 1]}$ be its convolution approximation, as defined in \eqref{eq_ConvApprox}. Then there exists a finite constant $C > 0$ such that, 
\begin{equation*}
\sup_{r \in (0, 1]} \sup_{x \in D_r} \E \left[e^{\sup_{u \in B(x,r)} X_r(u)-X_{r}(x)}\right] \leq C \,,
\end{equation*}
where $D_r := \{x \in D \, : \, \dist(x, \partial D) < r\}$. 
\end{lemma}
\begin{proof}
Fix $r \in (0,1]$ and $x \in D_r$. Consider the centred Gaussian field $(X_r^x(u))_{u \in B(x, r)}$ defined by
\begin{equation*}
X_r^x(u) := X_r(u) - X_r(x)\,,  \quad u \in B(x, r) \,,
\end{equation*}
and the random variable $\Omega_r^x$ given by
\begin{equation*}
\Omega_r^x := \sup_{u \in B(x, r)} X_r^x(u) \,. 
\end{equation*}
The variance of the field $(X_r^x(u))_{u \in B(x, r)}$ can be uniformly bounded in $r \in (0, 1]$ and $x \in D_r$. Indeed, for $u$, $v \in B(x, r)$, thanks to property \ref{pp_P3}, it exists a finite constant $K>0$ such that 
\begin{align}
\label{eq_UnifBoundCov}
\E[X_r^x(u)X_r^x(v))] & \leq -\log(|u-v|+r) + \log(|u-x|+r) + \log(|v-x|+r) - \log r + 4K \nonumber \\
& \leq -2\log r + 2\log 2r + 4K \nonumber \\
& = 2 \log 2 + 4K \,,
\end{align}
where the last inequality is due to the fact that $|u-x| \leq r$ and $|v-x| \leq r$. Thanks to property \ref{pp_P4}, we know that the the field $X_r^x$ is almost surely continuous and so we can apply Borell-TIS inequality (cf.\ Lemma~\ref{lm_Borell}) to obtain that
\begin{equation}
\label{eq_stepProofBoundLog}
\P\left(\left|\Omega_r^x- \E[\Omega_r^x ] \right| > t \right) \leq 2 e^{-\frac{t^2}{2 \sigma_{x, r}^2}} \,,
\end{equation}
for all $t \geq 0$, where $\sigma_{x, r}^2 := \sup_{u \in B(x, r)} \E \left[X_r^x(u)^2\right]$. We claim that both $\sigma_{x, r}^2$ and $\E[\Omega_r^x ]$ can be uniformly bounded in $x$ and $r$. For $\sigma_{x, r}^2$, this follows easily from \eqref{eq_UnifBoundCov}. For $\E[\Omega_r^x ]$, we can apply Dudley's entropy bound (cf.\ Lemma~\ref{lm_Dudley}). Indeed, doing computation similar to the one in \eqref{eq_UnifBoundCov}, and using the properties of the convolution approximation (cf.\ also \cite[Lemma~2.8]{Junnila_Imaginary}), one can easily verify that there exists a finite constant $A>0$, that does not depend on $x$ and $r$, such that  
\begin{equation*}
\E[(X_r^x(u) - X_r^x(v))^2] = \E[(X_r(u) - X_r(v))^2]  \leq A \frac{|u-v|}{r} \,, \quad \forall u, v \in B(x, r) \,. 
\end{equation*}
Therefore, applying Lemma~\ref{lm_Dudley} to the field $(X_r^x(u))_{u \in B(x,r)}$, one can readily obtain the desired uniform upper bound. Hence, from \eqref{eq_stepProofBoundLog}, it follows that there exist finite constants $a_0$, $a_1 > 0$, independent of $x$ and $r$, such that 
\begin{equation*}
\P\left(\Omega_r^x > t\right) \leq 2 e^{-\frac{(t-a_0)^2}{2 a_1}} \,.
\end{equation*}
Finally, remembering that for every non-negative random variable $X$ it holds that $\E[X] = \int_0^{\infty} \P(X > t) \d t$, we have that
\begin{equation*}
\E\left[e^{\Omega_r^x}\right] =  \int_0^{\infty} \P(e^{\Omega_r^x} > t) \d t = \int_0^{\infty} \P\left(\Omega_r^x > \log t \right) \d t \leq \int_0^{\infty} 2 e^{-\frac{(\log(t)-a_0)^2}{2 a_1}}\d t < \infty \,,
\end{equation*}
uniformly in $x$ and $r$, which concludes the proof.
\end{proof}

\subsection{Properties of Gaussian multiplicative chaos}
We recall here the main properties of GMC measures that we need in the following sections. For $d \geq 1$, let $D \subset \R^d$ be a bounded domain. Consider a log-correlated Gaussian field $X$ defined on $D$ with covariance kernel \eqref{eq_Covariance} and let $(X_{\eps})_{\eps \in (0, 1]}$ be its convolution approximation as defined in \eqref{eq_ConvApprox}. 

We start by recalling the following standard result concerning the existence and the non-degeneracy of GMC measures.
\begin{proposition}[{\cite[Theorem~1.1]{Berestycki_Elementary}}]
\label{pr_consGausChaos}
If $\gamma^2 < 2d$, then the sequence of approximating measures
\begin{equation*}
M^{\eps}_{\gamma} (\d x) := e^{\gamma X_{\eps}(x)-\frac{1}{2} \E[X_{\eps}(x)^2]} \d x \,, \quad \eps \in (0, 1] \,,
\end{equation*}
converges in probability and in $L^1(\P)$ to some non-degenerate limit $M_{\gamma}$, called GMC measure, in the space of Radon measures with respect to the topology of weak convergence. Furthermore, $M_{\gamma}$ does not depend on the choice of the mollifier $\psi$ used in \eqref{eq_ConvApprox}. 
\end{proposition}

For $\gamma^2 < 2d$, it can be easily verified that the random measure $M_{\gamma}$ is almost surely supported on the whole domain $D$ (cf.\ \cite{RV_Review}). On the other hand, $M_{\gamma}$ gives full measure to the set of $\gamma$-thick points $\mathcal{T}_{\gamma}$, defined in \eqref{eq_GammaThick}. It is known that the Hausdorff dimension of $\mathcal{T}_{\gamma}$ is equal to $d-\gamma^2/2$ (cf.\ \cite[Theorem~4.2]{RV_Review}). Immediate consequences of this fact are that $M_{\gamma}$ is almost surely atomless, and that $M_{\gamma}$ is singular with respect to $M_{\gamma'}$ for any $\gamma \neq \gamma'$. In particular, for $\gamma \neq 0$, $M_{\gamma}$ is singular with respect to the Lebesgue measure. Let us also observe that when $\gamma^2$ gets closer to $2d$, the measure $M_{\gamma}$ is carried by a set whose Hausdorff dimension gets closer to $0$. In particular, this means that $M_{\gamma}$ tends to cluster as $\gamma^2$ increases, until it degenerates at $\gamma ^2 = 2d$. Let us emphasize that there is a rich literature on critical GMC measures, i.e.\ when $\gamma^2 = 2d$, and we refer to \cite{Powell_Critical} for a review.

We have the following standard result concerning the existence of uniform bounds on positive and negative moments of GMC measures.
\begin{proposition}[{\cite[Theorems~2.11 and 2.12]{RV_Review}}]
\label{pr_UnifBound}
Let $\gamma^2 < 2d$ and $q < 2d/\gamma^2$. Then for any non-empty compact set $A \subset D$ it holds that
\begin{equation*}
\sup_{\eps \in (0,1]} \E\left[\left(\int_{A} e^{\gamma X_{\eps}(x)-\frac{1}{2} \gamma^2 \E[X_{\eps}(x)^2]}\d x \right)^{q}\right] < \infty \,.
\end{equation*}
\end{proposition}

An important feature of GMC measures is their multifractal behaviour. More precisely, it can be shown that the moments of $M_{\gamma}$ have a power law behaviour in which the exponents can be expressed through a non-linear function $\xi_{M_{\gamma}}$, called the power law spectrum, which is defined by 
\begin{equation}
\label{eq_powerSpec}
\xi_{M_{\gamma}}(q) : = \left(d+ \frac{1}{2} \gamma^2\right) q - \frac{1}{2} \gamma^2 q^2 \,, \quad q \in \R \,.
\end{equation}
More formally, we have the following proposition.
\begin{proposition}[{\cite[Theorem~2.14]{RV_Review}}]
\label{pr_PoweLawSpec}
Let $\gamma^2 < 2d$ and $q < 2d / \gamma^2$. Then there exists a finite constant $C > 0$ such that for all $x \in D$ and $r \in (0, 1]$ it holds that
\begin{equation*}
\E[M_{\gamma}(B(x, r))^q] \leq C r^{\xi_{M_{\gamma}}(q)} \,.
\end{equation*}
\end{proposition}

Finally, we state below a result on the local modulus of continuity of GMC measures which is an easy consequence of Proposition~\ref{pr_PoweLawSpec}.
\begin{proposition}
\label{pr_HolderGMC}
Let $\gamma^2 < 2d$ and $\eps \in (0, 1]$. Then, almost surely, there exist finite constants $C_1$, $C_2>0$ such that for all $x \in D$ and $r \in (0, 1]$ it holds that  
\begin{equation*}
C_1 r^{\left(\sqrt{d} + |\gamma|/\sqrt{2}\right)^2 + \eps} \leq M_{\gamma}(B(x, r)) \leq C_2 r^{\left(\sqrt{d} - |\gamma|/\sqrt{2}\right)^2 - \eps}
\end{equation*}
\end{proposition}
\begin{proof}
A proof of the upper bound is given in \cite[Theorem~2.2]{RVG_LBM} in the case $d=2$. However, the proof can be easily adapted to all dimensions. The lower bound can be proved similarly. Indeed, let us assume for simplicity that $D = (-1,1)^d$ and let $\alpha:= (\sqrt{d} + |\gamma|/\sqrt{2})^2$. For $n \in \mathbb{N}$, we write $\Sigma_{n} := 2^{-n} \mathbb{Z}^d \cap (-1,1)^d$ for the lattice $2^{-n} \mathbb{Z}^d$ restricted to the open box $(-1,1)^d$. Thanks to an union bound, Markov's inequality, and Proposition~\ref{pr_PoweLawSpec}, it holds for each $n \in \mathbb{N}$, $\eps \in (0, 1]$, and $q <0$ that
\begin{align*}
\P\left(\min_{z \in \Sigma_{n}} M_{\gamma}(B(z, 2^{-n})) \leq 2^{-n(\alpha + \eps)}\right) 
& = \P\left(\max_{z \in \Sigma_{n}} M_{\gamma}(B(z, 2^{-n}))^q \geq 2^{-nq(\alpha + \eps)}\right) \\
& \leq \sum_{z \in \Sigma_{n}} 2^{nq(\alpha+\epsilon)} \E[M_{\gamma}(B(z, 2^{-n})^q)]\\
& \lesssim 2^{n q \eps - n(\xi_{M_{\gamma}}(q) - d - \alpha q)} \,,
\end{align*}
where the implicit constant does not depend on $n$. If we chose $q = -\sqrt{2d}/|\gamma|$, then one can easily check that $\xi_{M_{\gamma}}(q) - d - \alpha q= 0$. Therefore, thanks to the Borel--Cantelli lemma, we have that it almost surely exists a finite constant $C >0$ such that 
\begin{equation*}
\min_{z \in \Sigma_{n}} M_{\gamma}(B(z, 2^{-n}) \geq C 2^{-n(\alpha+\eps)} \,, \quad \forall n \in \mathbb{N} \,.
\end{equation*}
To conclude, it is sufficient to notice that for all $x \in (-1,1)^d$ and $r \in (0, 1]$, there exist $n \in \mathbb{N}$ and $z \in \Sigma_{n+1}$ such that $2^{-n} < r \leq 2^{-n+1}$ and $B(z, 2^{-n-1}) \subset B(x,r)$.
\end{proof}

\subsection{Hausdorff dimension}
\label{subsec:Hausdorff}
We collect here the definition and some properties of the Hausdorff dimension that we need in the sequel. We refer to \cite{Falconer_Techniques} for further details.

\begin{definition}
Let $D \subset \R^d$, $d \geq 1$, and $s \geq 0$. We define the $s$-dimensional Hausdorff measure of $D$ by
\begin{equation*}
\mathcal{H}^{s}(D) := \lim_{\delta \searrow 0} \mathcal{H}^s_{\delta}(D) \,,
\end{equation*}
where 
\begin{equation*}
\mathcal{H}_{\delta}^{s}(D) := \inf\left\{\sum_{n \in \mathbb{N}} \diam(U_n)^s \, : \, (U_n)_{n \in \mathbb{N}} \text{ is a } \delta\text{-cover of } D\right\} \,, \quad \delta \in (0, 1] \,.
\end{equation*}
The Hausdorff dimension of $D$ is defined by the following equivalent formulas
\begin{equation*}
\dim_{\H}(D) := \inf\{s \geq 0 \, : \, \H^s(D) = 0 \} = \sup\{s \geq 0 \, : \, \H^s(D) = \infty\} \,,
\end{equation*}
with the convention that $\dim_{\H}(\emptyset) = 0$. 
\end{definition}
Note that $0 \leq \dim_{\H}(D) \leq d$ for any $D \subset \R^d$, $d \geq 1$. Moreover, from the definition of Hausdorff dimension, it immediately follows that if $\H^s(D) < \infty$, then $\dim_{\H}(D) \leq s$. Furthermore, the Hausdorff dimension enjoys monotonicity and countable stability, i.e.\ $\dim_{\H}(D) \leq \dim_{\H}(D')$ for any $D \subset D'$, and $\dim_{\H}(\cup_{n \in \mathbb{N}} D_n) = \sup_{n \in \mathbb{N}} \dim_{\H}(D_n)$ for any collection of subsets $(D_n)_{n \in \mathbb{N}}$. 

Other useful properties of the Hausdorff dimension are collected in the following proposition. 
\begin{proposition}[{\cite[Proposition~2.15]{Jackson}}]
\label{pr_HolderJackson}
Let $f: [0, \infty) \to \mathbb{R}$ be a continuous and increasing function. Consider a set $D \subset [0, \infty)$ and assume there exist finite constants $C$, $R$, $\alpha>0$ such that $|f(x+r) - f(x-r)| \leq C r^{\alpha}$ for all $r \in [0, R)$ and $x \in D$. Then it holds that 
\begin{equation*}
\dim_{\H}(f(D)) \leq \frac{1}{\alpha} \dim_{\H}(D) \,.
\end{equation*}
\end{proposition}

We conclude this subsection with the following result.
\begin{proposition}[{\cite[Proposition~2.3]{Falconer_Techniques}}]
\label{pr_FalconerHausdorffBound}
Let $\mu$ be a non-negative finite Radon measure on $\R^d$ supported on a bounded domain $D \subset \R^d$, $d \geq 1$. Let $E \subset D$ such that $\mu(E) > 0$. For each $x \in E$, we let $\dim_{\mu}(x)$ to be the local dimension of $\mu$ at $x$ as defined in \eqref{eq_localDimensionMeasure} below. If $\dim_{\mu}(x) = s$ for all $x \in E$, then $\dim_{\H}(E) = s$. 
\end{proposition}

\subsection{Multifractal analysis of measures}
We collect here some general facts about multifractal analysis of measures. Let $\mu$ be a non-negative finite Radon measure on $\R^d$, $d \geq 1$, which may be random or non-random, supported on a bounded domain $D \subset \R^d$.
\begin{definition}
The $L^q$-spectrum of $\mu$ is the function $\tau_{\mu} : \R \to \R$ defined by
\begin{equation}
\label{eq_LqSpectrum}
\tau_{\mu}(q) := \limsup_{r \searrow 0}\frac{\log \sup\sum_{i \in I} \mu(B(x_i, r))^q}{-\log r} \,, \quad q \in \R \,,
\end{equation}
where $(B(x_i , r))_{i \in I}$ is a countable family of disjoint closed balls with radius $r$ centred at $x_i \in D$, and the supremum is taken over all such families. Moreover, the Legendre--Fenchel transform of the $L^q$-spectrum $\tau_{\mu}$ is the function $\tau_{\mu}^{*}:[0, \infty) \to [0, \infty)$ defined by 
\begin{equation}
\label{eq_LegendreTransform}
\tau_{\mu}^{*}(\alpha) := \inf_{q \in \R}\left\{\alpha q + \tau_{\mu}(q)\right\} \wedge 0\,, \quad \alpha \in [0, \infty) \,.
\end{equation}
\end{definition}

We have the following classical result concerning the $L^q$-spectrum.
\begin{proposition}[{\cite[Section~3]{Ngai}}]
\label{pr_LauNgai}
The $L^{q}$-spectrum $\tau_{\mu}$ is a decreasing convex function with $\tau_{\mu}(1) = 0$. Moreover, $\dom(\tau_{\mu}) := \{q \in \R \, : \, \tau_{\mu}(q) < \infty\} = \R$ if and only if
\begin{equation*}
\limsup_{r \searrow 0} \frac{\log \inf_{x \in D}\{\mu(B(x, r))\}}{\log r} < \infty \,.
\end{equation*}
\end{proposition}

Let us collect here some notation. For $\alpha \geq 0$, we introduce the sets $\EE_{\mu}(\alpha)$, $\underline{\UU}_{\mu}(\alpha)$ and $\overline{\BB}_{\mu}(\alpha)$ defined by
\begin{gather}
\EE_{\mu}(\alpha) := \left\{x \in D \, : \, \dim_{\mu}(x)= \alpha\right\} \,, \label{eq_LevelSets} \\
\underline{\UU}_{\mu}(\alpha) := \{x \in D \, : \, \underline{\dim}_{\mu}(x) \leq \alpha \} \,, \quad \overline{\BB}_{\mu}(\alpha) := \{x \in D \, : \, \overline{\dim}_{\mu}(x) \geq \alpha \} \,, \label{eq_LevelSetsExtended}
\end{gather}
where
\begin{gather}
\dim_{\mu}(x) :=  \lim_{r \searrow 0} \frac{\log \mu(B(x, r))}{\log r} \,, \label{eq_localDimensionMeasure}  \\
\underline{\dim}_{\mu}(x) :=  \liminf_{r \searrow 0} \frac{\log \mu(B(x, r))}{\log r} \,, \quad \overline{\dim}_{\mu}(x) :=  \limsup_{r \searrow 0} \frac{\log \mu(B(x, r))}{\log r} \nonumber \,, 
\end{gather}
which are called the local dimension, the lower local dimension and the upper local dimension of $\mu$ at $x$, respectively. 
\begin{definition}
The singularity spectrum of $\mu$ is the function $\d_{\mu} : [0, \infty) \to [0, \infty)$ defined by
\begin{equation}
\label{eq_SingularitySpectrum}
\d_{\mu}(\alpha) := \dim_{\mathcal{H}}(\EE_{\mu}(\alpha)) \,, \quad \alpha \in [0, \infty) \,.
\end{equation}
\end{definition}

We have the following fundamental result.
\begin{proposition}
\label{pr_unionSet}
For $\alpha \geq 0$, it holds that $\dim_{\mathcal{H}}(\underline{\UU}_{\mu}(\alpha) \cap \overline{\BB}_{\mu}(\alpha)) \leq \tau_{\mu}^{*}(\alpha)$. 
\end{proposition}
A proof of Proposition~\ref{pr_unionSet}, in a slightly different setting, can be found in \cite[Theorem~1]{BrownEtAl}. However, the proof we provide here is a simple generalization of \cite[Proposition~A.1]{BarralComplex}. The proof is based on Besicovitch's covering theorem (cf.\ \cite{Falconer_Techniques}).
\begin{lemma}[Besicovitch's covering theorem]
\label{lem:Besicovitch}
Consider a set $E \subset \R^d$ and for each $x \in E$ fix a number $r_{x} > 0$ such that $\sup_{x \in E} r_{x} < \infty$. Then there exists an integer $\sigma_d$, depending only on the dimension $d$, for which there exist countable subfamilies $\mathcal{B}_1, \dots, \mathcal{B}_{\sigma_d}$ of $\{B(x, r_{x}) \, : \, x \in E\}$ such that $E \subset \cup_{i \in \{1, \dots, \sigma_d\}} \cup_{B \in \mathcal{B}_i} B$ and $\mathcal{B}_i$ is a collection of disjoint sets for each $i \in \{1, \dots, \sigma_d\}$. 
\end{lemma}
\begin{proof}[Proof of Proposition~\ref{pr_unionSet}]
Fix $\alpha \geq 0$. By definition \eqref{eq_LegendreTransform} of $\tau_{\mu}^{*}$, it is sufficient to prove that
\begin{equation}
\label{eq_mainBarral}
\dim_{\mathcal{H}}(\underline{\UU}_{\mu}(\alpha) \cap \overline{\BB}_{\mu}(\alpha)) \leq \alpha q + \tau_{\mu}(q) \,, \quad \forall q \in R \,.
\end{equation}
Without loss of generality, we can assume that $\tau_{\mu}(q) < \infty$, for all $q \in \R$. We split the proof into two parts. In the first part we prove inequality \eqref{eq_mainBarral} for $q \geq 0$, while in the second part we prove inequality \eqref{eq_mainBarral} for $q < 0$. 

\vspace*{2mm}
{\it Case $q \geq 0$.} Define $\overline{s}_{q, \eps} := (\alpha + \eps)q + \tau_{\mu}(q) + \eps$, for $\eps > 0$ fixed. Using the notation introduced in Subsection~\ref{subsec:Hausdorff}, it is sufficient to show that $\sup_{\delta \in (0, 1]}\mathcal{H}_{\delta}^{\overline{s}_{q, \eps}}(\underline{\UU}_{\mu}(\alpha) \cap \overline{\BB}_{\mu}(\alpha)) < \infty$. By definition, for every $x \in \underline{\UU}_{\mu}(\alpha) \cap \overline{\BB}_{\mu}(\alpha)$, it exists a decreasing sequence $(r_{x,n})_{n \in \mathbb{N}}$ converging to $0$ such that
\begin{equation*}
\mu(B(x, r_{x, n})) \geq r_{x, n}^{\alpha + \eps} \,, \quad \forall n \in \mathbb{N}\,.
\end{equation*} 
Fix $\delta \in (0, 1]$, then for each $x \in \underline{\UU}_{\mu}(\alpha) \cap \overline{\BB}_{\mu}(\alpha)$, we choose $n_x$ such that $r_{x, n_x} \in (0, \delta)$. For each $m \in \mathbb{N}$, we let 
\begin{equation*}
F_m := \left\{x \in \underline{\UU}_{\mu}(\alpha) \cap \overline{\BB}_{\mu}(\alpha) \, : \, 2^{-m} < r_{x, n_x} \leq 2^{-m+1}\right\} \,.
\end{equation*}
Thanks to the Besicovitch's covering theorem (cf.\ Lemma~\ref{lem:Besicovitch}), it exists a positive integer $\sigma_d$ such that for every $m \in \mathbb{N}$, we can find $\sigma_d$ disjoint subsets $F_{m,1}, \dots, F_{m, \sigma_d}$ of $F_m$ such that each $F_{m, j}$ is at most countable, the balls $B(x, r_{x, n_x})$ with centres at $x \in F_{m, j}$ are pairwise disjoint and 
\begin{equation*}
\left(\left(\left(B(x, r_{x, n_x})\right)_{x \in F_{m, j}}\right)_{j \in \{1, \dots, \sigma_d\}}\right)_{m \in \mathbb{N}}
\end{equation*}
is a $\delta$-cover of $\underline{\UU}_{\mu}(\alpha) \cap \overline{\BB}_{\mu}(\alpha)$. Then we have that
\begin{align*}
\mathcal{H}_{\delta}^{\overline{s}_{q, \eps}}(\underline{\UU}_{\mu}(\alpha) \cap \overline{\BB}_{\mu}(\alpha)) & \leq \sum_{m \in \mathbb{N}} \sum_{j = 1}^{\sigma_d} \sum_{x \in F_{m , j}} r_{x, n_x}^{(\alpha + \eps)q + \tau_{\mu}(q) + \eps} \\
& \leq \sum_{m \in \mathbb{N}} \sum_{j = 1}^{\sigma_d} \sum_{x \in F_{m , j}} \mu(B(x, r_{x, n_x}))^q r_{x, n_x}^{\tau_{\mu}(q) + \eps} \\
& \lesssim \sum_{m \in \mathbb{N}} \sum_{j = 1}^{\sigma_d} \sum_{x \in F_{m , j}} \mu(B(x, 2^{-m+1}))^q 2^{-m(\tau_{\mu}(q) + \eps)}  \,,
\end{align*}
where the implicit constant does not depend on $m$. We observe that, for all $j \in {1, \dots, \sigma_d}$, the family $(B(x, 2^{-m+1}))_{x \in F_{m, j}}$ can be divided into two countable families of disjoint closed balls with centres in $D$. Hence, by definition of $\tau_{\mu}$, for $m$ large enough, we get
\begin{equation*}
\sum_{x \in F_{m , j}} \mu(B(x, 2^{-m+1}))^q \lesssim 2^{m(\tau_{\mu}(q) + \eps/2)} \,,
\end{equation*}
which implies that $\sup_{\delta \in (0, 1]}\mathcal{H}_{\delta}^{\overline{s}_{q, \eps}}(\underline{\UU}_{\mu}(\alpha) \cap \overline{\BB}_{\mu}(\alpha)) < \infty$. 

\vspace*{2mm}
{\it Case $q < 0$.} Define $\underline{s}_{q, \eps} := (\alpha - \eps)q + \tau_{\mu}(q) + \eps$, for $\eps > 0$ fixed. As before, it is sufficient to show that $\sup_{\delta \in (0, 1]}\mathcal{H}_{\delta}^{\underline{s}_{q, \eps}}(\underline{\UU}_{\mu}(\alpha) \cap \overline{\BB}_{\mu}(\alpha)) < \infty$. By definition, for every $x \in \underline{\UU}_{\mu}(\alpha) \cap \overline{\BB}_{\mu}(\alpha)$, it exists a decreasing sequence $(u_{x,n})_{n \in \mathbb{N}}$ converging to $0$ such that
\begin{equation*}
\mu(B(x, u_{x, n})) \leq u_{x, n}^{\alpha - \eps} \,, \quad \forall n \in \mathbb{N} \,.
\end{equation*} 
Proceeding similarly to the previous case, we get
\begin{equation*}
\mathcal{H}_{\delta}^{\underline{s}_{q, \eps}}(\underline{\UU}_{\mu}(\alpha) \cap \overline{\BB}_{\mu}(\alpha)) \lesssim \sum_{m \in \mathbb{N}} \sum_{j = 1}^{\sigma_d} \sum_{x \in F_{m , j}} \mu(B(x, 2^{-m}))^q 2^{- m(\tau_{\mu}(q) + \eps)}  \,, 
\end{equation*}
where the implicit constant does not depend on $m$, and by definition of $\tau_{\mu}$, for $m$ large enough, we get
\begin{equation*}
\sum_{x \in F_{m , j}} \mu(B(x, 2^{-m}))^q \lesssim 2^{m(\tau_{\mu}(q) + \eps/2)} \,,
\end{equation*}
which implies that $\sup_{\delta \in (0, 1]}\mathcal{H}_{\delta}^{\underline{s}_{q, \eps}}(\underline{\UU}_{\mu}(\alpha) \cap \overline{\BB}_{\mu}(\alpha)) < \infty$. 
\end{proof}

Since $\EE_{\mu}(\alpha) \subset \underline{\UU}_{\mu}(\alpha) \cap \overline{\BB}_{\mu}(\alpha)$, an immediate consequence of Proposition~\ref{pr_unionSet} is the following corollary.
\begin{corollary}
\label{cr_BoundLqSpec}
For $\alpha \geq 0$, it holds that $\d_{\mu}(\alpha) \leq \tau_{\mu}^{*}(\alpha)$.
\end{corollary}
In particular, Corollary~\ref{cr_BoundLqSpec} implies that to prove the validity of the multifractal formalism for a non-negative Radon measure $\mu$, we only need to prove the bound $\d_{\mu}(\alpha) \geq \tau_{\mu}^{*}(\alpha)$ for all $\alpha \geq 0$.

\section{Multifractal analysis of Gaussian multiplicative chaos}
\label{sec:main}
Let us now turn to the multifractal analysis of GMC measures. For $d \geq 1$, let $D \subset \R^d$ be a bounded domain. Consider a log-correlated Gaussian field $X$ on $D$ with covariance kernel \eqref{eq_Covariance} and let $(X_{\eps})_{\eps \in (0, 1]}$ be its convolution approximation, as defined in \eqref{eq_ConvApprox}. For every $\gamma^2 < 2d$, consider the GMC measure $M_{\gamma}$ on $D$ associated with $X$, as defined in Proposition~\ref{pr_consGausChaos}. 

The main result of this article is contained in the following theorem. 
\begin{theorem}
\label{th_MainTheorem}
For $\gamma^2< 2d$, the GMC measure $M_{\gamma}$ satisfies the multifractal formalism. More precisely, for $\alpha \geq 0$ it holds that
\begin{equation}
\label{eq_MainMultiForm}
\d_{M_{\gamma}}(\alpha) = \tau^{*}_{M_{\gamma}}(\alpha) = \begin{cases}
d-\frac{1}{2}\left(\frac{d-\alpha}{\gamma} + \frac{\gamma}{2}\right)^2 \,, & \quad \text{ if } \alpha \in \left[\left(\sqrt{d} - \frac{|\gamma|}{\sqrt{2}}\right)^2,  \left(\sqrt{d} + \frac{|\gamma|}{\sqrt{2}}\right)^2 \right] \,, \\
0 \,, & \quad \text{ otherwise} \,,	
\end{cases}
\end{equation}
almost surely. Moreover, the $L^{q}$-spectrum of $M_{\gamma}$ is given by
\begin{equation}
\label{eq_MainFormLqSpec}
\tau_{M_{\gamma}}(q) =  \begin{cases}
- \xi'_{M_{\gamma}}(q_{-}) q \,, & \text{ if } q \in (-\infty, q_{-}]  \,, \\
d - \xi_{M_{\gamma}}(q)  \,, & \text{ if } q \in [q_{-}, q_{+}]  \,, \\
- \xi'_{M_{\gamma}}(q_{+}) q\,, & \text{ if } q \in [q_{+}, \infty) \,,
\end{cases} 
\end{equation}
almost surely, where $q_{\pm} := \pm \sqrt{2d}/|\gamma|$ and $\xi_{M_{\gamma}}$ is the power law spectrum of $M_{\gamma}$ as defined in \eqref{eq_powerSpec}.
\end{theorem}
\begin{remark}
We emphasize that the range of values of $\alpha \geq 0$ for which the set $\EE_{M_{\gamma}}(\alpha)$ is non-empty increases with $\gamma^2$. This is a consequence of the well known fact that when $\gamma^2$ gets closer to the critical value $2d$, the concentration of mass of $M_{\gamma}$ becomes more and more clustered and so it gives rise to a larger spectrum of singularities.
\end{remark}
\begin{remark}
\label{rm_LqSpectrum}
Since the GMC measure $M_{\gamma}$ is almost surely a non-negative finite Radon measure supported on a bounded domain, Proposition~\ref{pr_LauNgai} implies that the $L^{q}$-spectrum $\tau_{M_{\gamma}}$ is a decreasing convex function and $\tau_{M_{\gamma}}(1) = 0$. Moreover, thanks to the lower bound in Proposition~\ref{pr_HolderGMC}, one can easily verify that $\dom(\tau_{M_{\gamma}}) = \R$. 
\end{remark}

The rest of this section is devoted to the proof of Theorem~\ref{th_MainTheorem} which is based on the following result.
\begin{proposition} 
\label{pr_RVCarrier}
Let $\gamma^2 < 2d$ and $q^2 < 2d/\gamma^2$. For $M_{q \gamma}$-almost every $x \in D$, the GMC measure $M_{\gamma}$ satisfies
\begin{equation*}
\lim_{r \searrow 0} \frac{\log M_{\gamma}(B(x, r)) }{\log r} = d + \left(\frac{1}{2}-q\right)\gamma^2 \,,
\end{equation*}
almost surely.
\end{proposition}
\begin{proof}
The proof can be found in Appendix~\ref{sec:ProofProp}. 
\end{proof}

\begin{remark}
Roughly speaking, since for $q^2 < 2d/\gamma^2$ the measure $M_{q \gamma}$ is carried by the set of $q \gamma$-thick points of $X$, Proposition~\ref{pr_RVCarrier} implies that around such points the measure $M_{\gamma}$ has local dimension equal to 
\begin{equation}
\label{eq:alpha_q}
\alpha_q : = d + \left(\frac{1}{2}-q\right)\gamma^2 \,.
\end{equation}
\end{remark}

Thanks to Corollary~\ref{cr_BoundLqSpec}, we already know that, for $\alpha \geq 0$, it holds almost surely that $\d_{M_{\gamma}}(\alpha) \leq \tau_{M_{\gamma}}^{*}(\alpha)$. Therefore, to prove the validity of the multifractal formalism for $M_{\gamma}$, it suffices to show that for $\alpha \geq 0$ the lower bound $\d_{M_{\gamma}}(\alpha) \geq \tau_{M_{\gamma}}^{*}(\alpha)$ holds almost surely. Let us start with the following lemma which is an immediate consequence of Proposition~\ref{pr_RVCarrier}.
\begin{lemma}
\label{lm_RVCarrier}
For $\gamma^2 < 2d$, it holds almost surely that $\d_{M_{\gamma}}(\alpha_1) = \alpha_1$, where $\alpha_1$ is as in \eqref{eq:alpha_q} with $q=1$. Moreover, if $E \subset D$ is such that $\dim_{\mathcal{H}}(E) < \alpha_1$, then $M_{\gamma}(E) = 0$ almost surely.
\end{lemma}
\begin{proof}
For $x \in \EE_{M_{\gamma}}(\alpha_1)$, it obviously holds that $\dim_{M_{\gamma}}(x) =  \alpha_1$. Furthermore, thanks to Proposition~\ref{pr_RVCarrier}, we know that $M_{\gamma}(\EE_{M_{\gamma}}(\alpha_1)) > 0$ almost surely. Therefore, the first part of the lemma follows from Proposition~\ref{pr_FalconerHausdorffBound}. Concerning the second part of the lemma, let us proceed by contradiction. Consider a subset $E \subset D$ with $\dim_{\mathcal{H}}(E) < \alpha_1$ and such that $M_{\gamma}(E) > 0$. Thanks to Proposition~\ref{pr_RVCarrier}, we know that the subset $\EE_{M_{\gamma}}(\alpha_1)$ is of full $M_{\gamma}$-measure, and so we have that $M_{\gamma}(E \cap \EE_{M_{\gamma}}(\alpha_1)) = M_{\gamma}(E) > 0$. Moreover, since for all $x \in E \cap \EE_{M_{\gamma}}(\alpha_1)$ it holds that $\dim_{M_{\gamma}}(x) = \alpha_1$, then Proposition~\ref{pr_FalconerHausdorffBound} and the monotonicity of Hausdorff dimension imply that 
\begin{equation*}
\dim_{\H}(E ) \geq \dim_{\H}(E \cap \EE_{M_{\gamma}}(\alpha_1)) = \alpha_1 \,,
\end{equation*}
which is clearly a contradiction.
\end{proof}

Now, we introduce the structure function $\phi_{M_{\gamma}}$ associated with $M_{\gamma}$ which is defined by
\begin{equation}
\label{eq_StructureFunction}
\phi_{M_{\gamma}}(q) := d - \xi_{M_{\gamma}}(q) = \frac{1}{2} \gamma^2 q^2 - \left(d+\frac{1}{2} \gamma^2\right)q + d \,, \quad  q \in \R \,,
\end{equation} 
where we recall that $\xi_{M_{\gamma}}$ is the power law spectrum of $M_{\gamma}$ defined in \eqref{eq_powerSpec}. A direct computation shows that the Legendre--Fenchel transform of $\phi_{M_{\gamma}}$ can be written as follows
\begin{equation*}
\phi_{M_{\gamma}}^{*}(\alpha) = \left(d - \frac{1}{2}\left(\frac{d-\alpha}{\gamma}+ \frac{1}{2} \gamma\right)^2\right) \wedge 0  \,, \quad \alpha \geq 0 \,.
\end{equation*}
The following two lemmas provide a lower bound for $\d_{M_{\gamma}}(\alpha_{q})$ in terms of $\phi^{*}_{M_{\gamma}}(q)$, for all $q^2 < 2d/\gamma^2$, and an upper bound for $\tau_{M_{\gamma}}(q)$ in terms of $\phi_{M_{\gamma}}(q)$, for all $q \in \R$, respectively.
\begin{lemma}
\label{lm_dimComp}
Let $\gamma^2 < 2d$ and $q^2 < 2d/\gamma^2$. Then it holds almost surely that $\d_{M_{\gamma}}(\alpha_{q}) \geq \phi_{M_{\gamma}}^{*}(\alpha_{q})$, where $\alpha_{q}$ is as defined in \eqref{eq:alpha_q}.
\end{lemma}
\begin{proof}
On one hand, thanks to Proposition~\ref{pr_RVCarrier}, we know that the set $\EE_{M_{\gamma}}(\alpha_{q})$ is almost surely of full $M_{q \gamma}$-measure. On the other hand, Lemma~\ref{lm_RVCarrier} implies that $M_{q \gamma}$ cannot give positive measure to a set of Hausdorff dimension strictly less than $d-q^2 \gamma^2/2$. Therefore, it should hold almost surely that
\begin{equation*}
\d_{M_{\gamma}}(\alpha_q) \geq d-\frac{1}{2} q^2 \gamma^2 = \phi_{M_{\gamma}}^{*}(\alpha_q) \,,
\end{equation*}
where the last equality can be checked directly.
\end{proof}

\begin{lemma}
\label{lm_IneqLqSpec}
For $\gamma^2 < 2d$, it holds almost surely that $\tau_{M_{\gamma}}(q) \leq \phi_{M_{\gamma}}(q)$ for all $q \in \R$. 
\end{lemma}
\begin{proof}
Thanks to Remark~\ref{rm_LqSpectrum}, the function $\tau_{M_{\gamma}} : \R \to \R$ is convex and hence continuous. Therefore, it is sufficient to prove that for all $q \in \R$, it holds almost surely that $\tau_{M_{\gamma}}(q) \leq \phi_{M_{\gamma}}(q)$. For simplicity, we assume that $D = (-1,1)^d$. We split the proof into two parts. In the first part we focus on $q \geq 2d/\gamma^2$, while in the second part we focus on $q < 2d/\gamma^2$. 

\vspace*{2mm}
{\it Case $q \geq 2d/\gamma^2 $.} Since $2d/\gamma^2>1$, the conclusion in this case is trivial. Indeed, since $\tau_{M_{\gamma}}$ is a decreasing function and $\tau_{M_{\gamma}}(1) = 0$ (cf.\ Remark~\ref{rm_LqSpectrum}), it holds almost surely that $\tau_{M_{\gamma}}(q) \leq 0 \leq \phi_{M_{\gamma}}(q)$, where the last inequality can be checked directly. 

\vspace*{2mm}
{\it Case $q < 2d/\gamma^2$.} For $n \in \mathbb{N}$, we write $\Sigma_{n} := 2^{-n-1}\mathbb{Z}^d \cap (-1,1)^d$ for the lattice $2^{-n-1}\mathbb{Z}^d$ restricted to the open box $(-1,1)^d$. Then for each $n \in \mathbb{N}$ and $\eps > 0$, by Markov's inequality and Proposition~\ref{pr_PoweLawSpec}, we have that
\begin{align*}
\P\left(\sum_{z \in \Sigma_{n}} M_{\gamma}(B(z, 2^{-n}))^{q} \geq 2^{n(\phi_{M_{\gamma}}(q)+\eps)}\right) & \leq 2^{-n(\phi_{M_{\gamma}}(q)+\eps)} \sum_{z \in \Sigma_{n}} \E\left[M_{\gamma}(B(z, 2^{-n}))^{q} \right] \\
&\lesssim 2^{-n(\phi_{M_{\gamma}}(q)+\eps)} 2^{-n \xi_{M_{\gamma}}(q)} 2^{nd} = 2^{-n \eps} \,,
\end{align*}
where the implicit constant does not depend on $n$. Hence, it holds that
\begin{equation*}
\sum_{n \in \mathbb{N}} \P\left(\sum_{z \in \Sigma_{n}} M_{\gamma}(B(z, 2^{-n}))^{q} \geq 2^{n(\phi_{M_{\gamma}}(q)+\eps)}\right) < \infty \,,
\end{equation*}
and so, thanks to the Borel--Cantelli lemma, we have that it almost surely exists a finite constant $C >0$ independent of $n$ such that 
\begin{equation}
\label{eq_mainEqMainLemma}
\sum_{z \in \Sigma_{n}}  M_{\gamma}(B(z, 2^{-n}))^{q} \leq C 2^{n(\phi_{M_{\gamma}}(q)+\eps)} \,, \quad \forall n \in \mathbb{N} \,.
\end{equation}
Let $(B(x_i , r))_{i \in I}$ be a countable family of disjoint closed balls with radius $r \in(0,1]$ centred at $x_i \in (-1,1)^d$, and fix $n \in \mathbb{N}$ such that $2^{-n} < r \leq 2^{-n+1}$. We observe that there exists a finite constant $\sigma_d >0$ depending only on the dimension $d$, such that every ball $B(x_i, r)$ is contained in the union of at most $\sigma_d$ balls of the form $B(z_{ij}, 2^{-n+1})$, for some $z_{ij} \in \Sigma_{n-1}$, $j = 1, \dots, \sigma_d$. Here, we adopt the convention that if for $i \in I$, the ball $B(x_i, r)$ is contained in $\cup_{j=1, \dots, J} B(z_{ij}, 2^{-n+1})$ for some $J < \sigma_d$, then $z_{ij} = z_{iJ} $ for $j = J+1, \dots, \sigma_d$. Furthermore, since the balls $(B(x_i , r))_{i \in I}$ are disjoint, we can choose the balls $(B(z_{ij} , 2^{-n+1}))_{i \in I, j \in {1, \dots, \sigma_d}}$ in such a way that for fixed $i_1 \neq i_2 \in I$, it holds that $z_{i_1 j} \neq {z_{i_2 j}}$ for $j = 1, \dots, \sigma_d$. Now, if $q \in (0,1]$, thanks to \eqref{eq_mainEqMainLemma} and the sub-additivity of the function $x \mapsto x^{q}$, it holds almost surely that 
\begin{equation}
\label{eq_IneqLqSpec1}
\sum_{i \in I} M_{\gamma}(B(x_i, r))^q \leq \sum_{j = 1, .., \sigma_d}  \sum_{z \in \Sigma_{n-1}} M_{\gamma}(B(z, 2^{-n+1}))^q \lesssim 2^{n(\phi_{M_{\gamma}}(q)+\eps)} \lesssim r^{-(\phi_{M_{\gamma}}(q) + \eps)}  \,,
\end{equation}
where the implicit constant does not depend on $n$. Similarly, if $q \in (1, 2d/\gamma^2)$, thanks to \eqref{eq_mainEqMainLemma} and the convexity of the function $x \mapsto x^{q}$ for $q>1$, using Jensen's inequality, it holds almost surely that
\begin{equation}
\label{eq_IneqLqSpec2}
\sum_{i \in I} M_{\gamma}(B(x_i, r))^q \leq \sigma_d^{q-1}\sum_{j = 1, .., \sigma_d}  \sum_{z \in \Sigma_{n-1}} M_{\gamma}(B(z, 2^{-n+1}))^q  \lesssim 2^{n(\phi_{M_{\gamma}}(q)+\eps)} \lesssim r^{-(\phi_{M_{\gamma}}(q) + \eps)}  \,,
\end{equation}
where the implicit constant does not depend on $n$. Moreover, for every ball $B(x_i, r)$, we have that $B(z, 2^{-n-1}) \subset B(x_i, r)$ for some $z \in \Sigma_{n+1}$. Therefore, if $q \leq 0$, then thanks to \eqref{eq_mainEqMainLemma} it holds almost surely that 
\begin{equation}
\label{eq_IneqLqSpec3}
\sum_{i \in I} M_{\gamma}(B(x_i, r))^q \leq \sum_{z \in \Sigma_{n+1}} M_{\gamma}(B(z, 2^{-n-1}))^q \lesssim 2^{n(\phi_{M_{\gamma}}(q)+\eps)} \lesssim  r^{-(\phi_{M_{\gamma}}(q) + \eps)}  \,,
\end{equation}
where the implicit constant does not depend on $n$. Finally, the conclusion follows by taking logs in \eqref{eq_IneqLqSpec1}, \eqref{eq_IneqLqSpec2}, \eqref{eq_IneqLqSpec3}, and thanks to the arbitrariness of $\eps > 0$.
\end{proof}

Thanks to Corollary~\ref{cr_BoundLqSpec}, Lemmas~\ref{lm_dimComp} and \ref{lm_IneqLqSpec}, we obtain that for $q^2 < 2d/\gamma^2$ it holds almost surely that
\begin{equation}
\label{eq_MainInequality}
0 < \phi_{M_{\gamma}}^{*}(\alpha_{q}) \leq \d_{M_{\gamma}}(\alpha_{q}) \leq \tau_{M_{\gamma}}^{*}(\alpha_q) \leq \phi_{M_{\gamma}}^{*}(\alpha_{q}) \,,
\end{equation}
which implies that $\d_{M_{\gamma}}(\alpha_{q}) = \tau_{M_{\gamma}}^{*}(\alpha_q)$ and $\tau_{M_{\gamma}}(q) = \phi_{M_{\gamma}}(q)$.  For $q^2 < 2d/\gamma^2$, we observe that the range of values of $\alpha \geq 0$ of the form $\alpha_q$ is equal to the interval $(\alpha_{-}, \alpha_{+})$, where
\begin{equation*}
\alpha_{-} = \left(\sqrt{d} - \frac{|\gamma|}{\sqrt{2}}\right)^2 \,, \quad \alpha_{+} = \left(\sqrt{d} + \frac{|\gamma|}{\sqrt{2}}\right)^2  \,.
\end{equation*}

Thanks to Corollary~\ref{cr_BoundLqSpec} and Lemma~\ref{lm_IneqLqSpec}, we can easily observe that for $\alpha \not \in (\alpha_{-}, \alpha_{+})$ we have that $\d_{M_{\gamma}}(\alpha) \leq \tau^{*}_{M_{\gamma}}(\alpha) \leq \phi^{*}_{M_{\gamma}}(\alpha) = 0$. Therefore, these results imply that for all $\alpha \geq 0$, it holds almost surely that 
\begin{equation*}
\d_{M_{\gamma}}(\alpha) =\tau_{M_{\gamma}}^{*}(\alpha) = \begin{cases}
d-\frac{1}{2}\left(\frac{d-\alpha}{\gamma} + \frac{\gamma}{2}\right)^2 \,, & \quad \text{ if } \alpha \in \left[\left(\sqrt{d} - \frac{|\gamma|}{\sqrt{2}}\right)^2,  \left(\sqrt{d} + \frac{|\gamma|}{\sqrt{2}}\right)^2 \right] \,, \\
0 \,, & \quad \text{ otherwise} \,,	
\end{cases}
\end{equation*}
which coincides with \eqref{eq_MainMultiForm}.

To conclude the proof of Theorem~\ref{th_MainTheorem} we need to show the validity of \eqref{eq_MainFormLqSpec}. This is exactly the objective of the subsequent lemma whose proof follows similar steps of those in \cite[Section~2.3]{BarralComplex}.
\begin{lemma}
\label{lm_LqSpecFinal}
If $\gamma^2 < 2d$, then for $q \in \R$ it holds almost surely that  
\begin{equation*}
\tau_{M_{\gamma}}(q) = \begin{cases}
-\xi'_{M_{\gamma}}(q_{-})q  \,, & \quad \ \text{ if } q \in(-\infty, q_{-}] \,, \\
d-\xi_{M_{\gamma}}(q) \,, & \quad \ \text{ if } q \in [q_{-},  q_{+}] \,, \\
-\xi'_{M_{\gamma}}(q_{+})q \,, & \quad \ \text{ if } q \in [q_{+}, \infty) \,,
\end{cases} 
\end{equation*}
where $q_{\pm} = \pm \sqrt{2d}/|\gamma|$.
\end{lemma}
\begin{proof}
From \eqref{eq_MainInequality}, we know that $\tau_{M_{\gamma}}(q) = \phi_{M_{\gamma}}(q)$ for $q \in (q_{-}, q_{+})$. Thanks to the continuity of the convex function $\tau_{M_{\gamma}}$, the above equality can be extended to the close interval $[q_{-}, q_{+}]$. Therefore, it remains to prove that $\tau_{M_{\gamma}}$ is derivable at $q_{+}$ and linear on $[q_{+}, \infty)$, and the same also for $q_{-}$ and $(-\infty, q_{-}]$. Thanks to the equality $\tau_{M_{\gamma}}(q) = \phi_{M_{\gamma}}(q)$ on the interval $[q_{-}, q_{+}]$, we know that
\begin{equation*}
\tau'_{M_{\gamma}}(q_{+}^{-}) =\phi'_{M_{\gamma}}(q_{+}) = \frac{\phi_{M_{\gamma}}(q_{+})}{q_{+}}\,, \quad \tau'_{M_{\gamma}}(q_{-}^{+}) =\phi'_{M_{\gamma}}(q_{-}) = \frac{\phi_{M_{\gamma}}(q_{-})}{q_{-}} \,,
\end{equation*}
where $\tau'_{M_{\gamma}}(q_{+}^{-})$ (resp.\ $\tau'_{M_{\gamma}}(q_{-}^{+})$) denotes the left-hand (resp.\ right-hand) derivative of $\tau_{M_{\gamma}}$ at $q_{+}$ (resp.\ at $q_{-}$). Since the function $\tau_{M_{\gamma}}$ is convex, we have that
\begin{equation*}
\tau_{M_{\gamma}}(q) \geq \tau_{M_{\gamma}}(q_{+}) + \tau'_{M_{\gamma}}(q_{+}^{-})(q-q_{+}) = \phi'_{M_{\gamma}}(q_{+})q \,,
\end{equation*}
and similarly $\tau_{M_{\gamma}}(q) \geq \phi'_{M_{\gamma}}(q_{-})q$. For $q \geq q_{+}$, the reverse inequality can be obtained using the fact that $q_{+}/q  \in (0, 1]$ and the sub-additivity of the function $x \mapsto x^{q_{+}/q }$. Indeed, if we consider a countable family $(B(x_i , r))_{i \in I}$ of disjoint closed balls with radius $r$ centred at $x_i \in D$, then
\begin{equation*}
\left(\sum_{i \in I} M_{\gamma}(B(x_i, r))^{q}\right)^{\frac{q_{+}}{q}} \leq \sum_{i \in I} M_{\gamma}(B(x_i, r))^{q_{+}}\,.
\end{equation*}
Hence, taking logs and by definition of $\tau_{M_{\gamma}}$, we have that
\begin{equation*}
\tau_{M_{\gamma}}(q) \leq \frac{\tau_{M_{\gamma}}(q_{+})}{q_{+}} q  = \frac{\phi_{M_{\gamma}}(q_{+})}{q_{+}} q = \phi'_{M_{\gamma}}(q_{+})q \,.
\end{equation*} 
Therefore, for all $q \geq q_{+}$ it holds that $\tau_{M_{\gamma}}(q) = \phi'_{M_{\gamma}}(q_{+})q$, and the same exact computation yields also that $\tau_{M_{\gamma}}(q) = \phi'_{M_{\gamma}}(q_{-})q$ for all $q \leq q_{-}$. Since $\phi_{M_{\gamma}}(q) = d - \xi_{M_{\gamma}}(q)$, for all $q \in \R$, the result is proved.
\end{proof}

This concludes the proof of Theorem~\ref{th_MainTheorem}.

\section{Multifractal analysis of multifractal random walk}
\label{sec:MRW}
Theorem~\ref{th_MainTheorem} gives a full characterization of the singularity spectrum of GMC measures. A natural question that may arise is whether we can use this result to perform the multifractal analysis of objects that are built by means of these random measures. In particular, in this section, we investigate the multifractal behaviour of the \emph{Multifractal Random Walk} (MRW).

\subsection{Definition of the MRW}
Let $T > 0$ and consider a log-correlated Gaussian field $X$ on $[0, T]$ with covariance kernel of the from \eqref{eq_Covariance}. For $\gamma^2 < 2$, consider the GMC measure $M_{\gamma}$ on $[0, T]$ associated with $X$, as defined in Proposition~\ref{pr_consGausChaos}. In addition, for $d \geq 1$, let $B$ be a Brownian motion on $\R^d$ independent from the field $X$.
\begin{remark}
\label{rem_ConventionProb}
As the Brownian motion $B$ is independent from the log-correlated Gaussian field $X$, it will be useful to decompose the probability measure $\P$ as follows
\begin{equation*}
\P = \P_B \otimes \P_X \,.
\end{equation*}
When we write ``almost surely'', we mean $\P$-almost surely. If we need to specify that a property holds almost surely with respect to the Brownian motion (resp.\ the field $X$), we write $\P_B$-almost surely (resp.\ $\P_X$-almost surely).
\end{remark}
\begin{definition}
Fix $\gamma^2 < 2$ and $d \geq 1$. The $d$-dimensional MRW on $[0, T]$ associated with the field $X$ is the stochastic process $(\Z^{\gamma}_t)_{t \in [0, T]}$ defined by 
\begin{equation}
\label{eq_MRW}
\Z^{\gamma}_t := B_{M_{\gamma}([0, t])} \,, \quad t \in [0, T]\,.
\end{equation}
\end{definition}

The MRW has been originally introduced in \cite{MRW} as a stochastic volatility model. Indeed, assuming that the Brownian motion $B$ is one-dimensional, the MRW is used for modelling the price fluctuations of financial assets since it accounts for most of the stylized facts of financial time series. For example, it reproduces the sudden jumps and the periods of intense activity usually observed in stock prices. We refer to \cite{MRWAsset} for a short overview on this topic. 

\subsection{Study of the lower singularity spectrum}
The main goal of this section is to analyse the local fluctuations of the paths of the MRW $\Z^{\gamma}$. More precisely, we introduce for $\alpha \geq 0$ the sets
\begin{equation*}
\underline{\EE}_{\Z^{\gamma}}(\alpha) := \left\{t \in [0, T] \, : \, \underline{\dim}_{\Z^{\gamma}}(t)= \alpha\right\} \,, 
\end{equation*}
where $\underline{\dim}_{\Z^{\gamma}}(t)$ is the lower local dimension of $\Z^{\gamma}$ at $t \in [0, T]$ defined by
\begin{equation*}
\underline{\dim}_{\Z^{\gamma}}(t) :=\liminf_{r \searrow 0} \frac{\log |\Z^{\gamma}_{t+r} - \Z^{\gamma}_{t-r}| }{\log r} \,.
\end{equation*}
We want to compute the Hausdorff dimension of $\underline{\EE}_{\Z^{\gamma}}(\alpha)$ for each $\alpha \geq 0$, i.e.\ we are interested in finding an explicit expression for the lower singularity spectrum $\underline{\d}_{\Z^{\gamma}} : [0, \infty) \to  [0, \infty)$ defined by 
\begin{equation*}
\underline{\d}_{\Z^{\gamma}}(\alpha) := \dim_{\mathcal{H}}(\underline{\EE}_{\Z^{\gamma}}(\alpha)) \,, \quad \alpha \in [0, \infty) \,.
\end{equation*}
In particular, we want to find the relation between $\underline{\d}_{\Z^{\gamma}}(\alpha)$ and $\d_{M_{\gamma}}(\alpha)$. This is exactly the content of the subsequent theorem.
\begin{theorem}
\label{th_MainTheoremBMMT}
Let $\gamma^2 < 2$ and $d \geq 1$. Then, for $\alpha \geq 0$, the lower singularity spectrum of the $d$-dimensional MRW $\Z^{\gamma}$ is given by
\begin{equation*}
\underline{\d}_{\Z^{\gamma}}(\alpha) = \d_{M_{\gamma}}(2 \alpha) = \begin{cases}
1-\frac{1}{2}\left(\frac{1-2 \alpha}{\gamma} + \frac{\gamma}{2}\right)^2 \,, & \quad \text{ if } \alpha \in \left[\left(\frac{1}{\sqrt{2}} - \frac{|\gamma|}{2}\right)^2,  \left(\frac{1}{\sqrt{2}} + \frac{|\gamma|}{2}\right)^2 \right] \,, \\
0 \,, & \quad \text{ otherwise} \,,
\end{cases} 
\end{equation*}
almost surely.
\end{theorem}

For $\alpha \geq 0$, we define the set $\EE_{M_{\gamma}}(\alpha)$ as in \eqref{eq_LevelSets}, and the sets $\underline{\UU}_{M_{\gamma}}(\alpha)$, $\overline{\BB}_{M_{\gamma}}(\alpha)$ as in \eqref{eq_LevelSetsExtended}. Then Theorem~\ref{th_MainTheoremBMMT} is an easy consequence of the following lemma.
\begin{lemma}
\label{lm_dimBM}
If $\gamma^2 < 2$, then for $\alpha \geq 0$, it holds $\P_B$-almost surely that
\begin{enumerate}[label=(\roman*)]
\itemsep0em 
\item $\EE_{M_{\gamma}}(2 \alpha) \subset \underline{\EE}_{\Z^{\gamma}}(\alpha)$;
\item $\underline{\EE}_{\Z^{\gamma}}(\alpha) \subset \underline{\UU}_{M_{\gamma}}(2 \alpha) \cap  
\overline{\BB}_{M_{\gamma}}(2 \alpha)$.
\end{enumerate}
\end{lemma}
\begin{proof}
Let $(\Z^{\gamma}_t)_{t \in [0, T]}$ be the $d$-dimensional MRW as defined in \eqref{eq_MRW}. Thanks to well-known properties of Brownian motion, the following facts holds $\P_B$-almost surely. Fix $\eps > 0$, then it exists a finite constant $C > 0$ such that, for every $t \in [0, T]$ and $r >0$ small enough, it holds that 
\begin{equation}
\label{eq_Levy1}
|\Z^{\gamma}_{t+r} - \Z^{\gamma}_{t-r}| = |B_{M_{\gamma}([0, t+r])}-B_{M_{\gamma}([0, t-r])}| \leq C M_{\gamma}(B(t, r))^{\frac{1}{2}-\eps} \,.
\end{equation}
Moreover, we also know that for each $\eps > 0$ and $t \in [0, T]$, for $h > 0$ small enough, it holds that   
\begin{equation}
\label{eq_Levy2}
\sup_{0 < r \leq h} |\Z^{\gamma}_{t+r} - \Z^{\gamma}_{t-r}| = \sup_{0 < r \leq h} |B_{M_{\gamma}([0, t+r])}-B_{M_{\gamma}([0, t-r])}| \geq M_{\gamma}(B(t, r))^{\frac{1}{2} + \eps} \,.
\end{equation}

\vspace*{2mm}
{\it Proof of item (i).} For $\alpha \geq 0$, we fix $t \in \EE_{M_{\gamma}}(2\alpha)$. For every $\eps > 0$, we know that for $r > 0$ small enough, it holds that
\begin{equation}
\label{eq_Levy3}
r^{2\alpha+\eps} \leq M_{\gamma}(B(t, r)) \leq r^{2\alpha-\eps} \,.
\end{equation}
Consequently, thanks to \eqref{eq_Levy1} and the right-hand side of \eqref{eq_Levy3}, we get that for $r > 0$ small enough, it holds that
\begin{equation*}
|\Z^{\gamma}_{t+r}-\Z^{\gamma}_{t-r}| \leq C M_{\gamma}(B(t, r))^{\frac{1}{2}-\eps} \leq C r^{(2\alpha-\eps)\left(\frac{1}{2}-\eps\right)} \,.
\end{equation*}
On the other hand, thanks to \eqref{eq_Levy2} and the left-hand side of \eqref{eq_Levy3}, there exists a decreasing sequence $(r_n)_{n \in \mathbb{N}}$ converging to $0$ such that 
\begin{equation*}
|\Z^{\gamma}_{t+r_n}-\Z^{\gamma}_{t-r_n}| \geq M_{\gamma}(B(t, r_n))^{\frac{1}{2}+\eps} \geq r_n^{(2\alpha+\eps)\left(\frac{1}{2}+\eps\right)} \,, \quad \forall n \in \mathbb{N} \,.
\end{equation*}
Therefore, taking logs, by arbitrariness of $\eps >0$, this implies that 
\begin{equation*}
\liminf_{r \searrow 0}  \frac{\log |\Z^{\gamma}_{t+r}-\Z^{\gamma}_{t-r}|}{\log r} =  \alpha \,.
\end{equation*}
which means that $t \in \underline{\EE}_{\Z^{\gamma}}(\alpha)$, $\P_B$-almost surely.

\vspace*{2mm}
{\it Proof of item (ii).} For $\alpha \geq 0$, we fix $t \in \underline{\EE}_{\Z^{\gamma}}(\alpha)$. Similarly as above, for each $\eps > 0$, there exist decreasing sequences $(r_n)_{n \in \mathbb{N}}$, $(r'_n)_{n \in \mathbb{N}}$ both converging to $0$ such that 
\begin{equation*}
r_n^{\alpha+\eps} \leq |\Z^{\gamma}_{t+r_n}-\Z^{\gamma}_{t-r_n}| \leq C M_{\gamma}(B(t, r_n))^{\frac{1}{2}-\eps}\,, \quad \forall n \in \mathbb{N} \,,
\end{equation*}
where the last inequality is given by \eqref{eq_Levy1}, and
\begin{equation*}
{r'_n}^{\alpha-\eps} \geq |\Z^{\gamma}_{t+r'_n} - \Z^{\gamma}_{t-r'_n}| \geq M_{\gamma}(B(t, r'_n))^{\frac{1}{2}+\eps} \,, \quad \forall n \in \mathbb{N} \,,
\end{equation*}
where the last inequality is given by \eqref{eq_Levy2}. By arbitrariness of $\eps > 0$, these facts imply that $t \in \underline{\UU}_{M_{\gamma}}(2\alpha) \cap \overline{\BB}_{M_{\gamma}}(2 \alpha)$, $\P_B$-almost surely.
\end{proof}

Finally, we show how the proof of Theorem~\ref{th_MainTheoremBMMT} follows immediately from  Lemma~\ref{lm_dimBM}.
\begin{proof}[Proof of Theorem~\ref{th_MainTheoremBMMT}]
From item (i) of Lemma~\ref{lm_dimBM}, for all $\alpha \geq 0$, we have the inequality 
\begin{equation*}
\underline{\d}_{\Z^{\gamma}}(\alpha) \geq \d_{M_{\gamma}}(2 \alpha)\,,
\end{equation*}
almost surely. On the other hand, item (ii) of Lemma~\ref{lm_dimBM}, Proposition~\ref{pr_unionSet} and Theorem~\ref{th_MainTheorem} imply that, for all $\alpha \geq 0$, it holds
\begin{equation*}
\d_{\Z^{\gamma}}(\alpha) \leq \dim_{\mathcal{H}} \left(\underline{\UU}_{M_{\gamma}}(2 \alpha) \cap \overline{\BB}_{M_{\gamma}}(2 \alpha) \right) \leq \tau^{*}_{M_{\gamma}} (2 \alpha) = \d_{M_{\gamma}}(2 \alpha) \,,
\end{equation*}
almost surely. Since Theorem~\ref{th_MainTheorem} gives a full characterization of the singularity spectrum $\d_{M_{\gamma}}$, this concludes the proof.
\end{proof}

\section{Multifractal analysis of Liouville Brownian motion}
\label{sec:LBM}
The main goal of this section is to study the multifractal behaviour of the \emph{Liouville Brownian motion} (LBM). Roughly speaking, the LBM can be defined as the canonical diffusion associated with the \emph{Liouville quantum gravity} metric tensor given by
\begin{equation*}
e^{\gamma X} (\d x^2 + \d y^2) \,,
\end{equation*}
where $X$ is a GFF on a domain $D \subset \R^2$ and $\d x^2 +\d y^2$ is the Euclidean metric tensor. The LBM has been introduced simultaneously in \cite{Ber_LBM, RVG_LBM}, and it can be built as a time changed planar Brownian motion with respect to a GMC measure associated with the field $X$. We refer to the above mentioned references for more details.

In what follows, we take a slightly more general approach, and we generalize the construction of the LBM in dimension $d \geq 2$. Once this is done, we consider the problem of studying the lower singularity spectrum of the $d$-dimensional LBM. 

\subsection{Definition of the LBM}
For $d \geq 2$, consider a bounded domain $D \subset \R^d$ containing the ball $B(0, 1)$. Let $B$ be a $d$-dimensional Brownian motion started from the origin and define $T$ to be its first exit time from $D$, i.e.\
\begin{equation}
\label{eq_exitTime}
T := \inf\left\{t \geq 0 \, : \, B_t \not \in D\right\} \,.
\end{equation}
Moreover, consider a log-correlated Gaussian field $X$ on $D$ with covariance kernel of the form \eqref{eq_Covariance} and independent from the Brownian motion $B$. Then the $d$-dimensional LBM on $D$ can be formally defined by
\begin{equation*}
\B^{\gamma}_t := B_{F_{\gamma}^{-1}(t)} \,, \quad t \geq 0 \,,
\end{equation*}
where the clock process $F_{\gamma}$ is given by
\begin{equation*}
F_{\gamma}(t) :=\int_0^{t \wedge T} e^{\gamma X(B_s) - \frac{1}{2} \gamma^2 \E[X(B_s)^2]} \d s \,, \quad t \geq 0 \,.
\end{equation*}
In the section, we adopt the same notation specified in Remark~\ref{rem_ConventionProb}, i.e., we decompose $\P$ as  $\P_B \otimes \P_X$.

To give rigorous sense to the definition of the LBM, it is sufficient to show the existence of the clock process $F_{\gamma}$. Following \cite{Ber_LBM, RVG_LBM}, we proceed with a regularization and limiting procedure. Let $(X_{\eps})_{\eps \in (0, 1]}$ be the convolution approximation of the field $X$ as defined in \eqref{eq_ConvApprox}, then we have the following lemma.
\begin{lemma}
\label{lm_ConstructionLBM}
If $\gamma^2 < 4$, then $\P_B$-almost surely, for all $t \geq 0$, the sequence
\begin{equation*}
F_{\gamma}^{\eps}(t) := \int_0^{t \wedge T} e^{\gamma X_{\eps} (B_s) - \frac{1}{2} \gamma^2 \E[X_{\eps} (B_s)^2]} \d s\,, \quad \eps \in (0, 1] \,,
\end{equation*}
converges in $\P_X$-probability and in $L^1(\P_X)$ towards a non-degenerate limit $F_{\gamma}(t)$ that does not depend on the choice of the mollifier $\psi$ used in \eqref{eq_ConvApprox}. 
\end{lemma}
\begin{proof}
This fact is a standard result and it is an easy consequence of \cite[Theorem~1.1]{Berestycki_Elementary}. Indeed, for $t \geq 0$, let $\nu^T_{t}$ be the occupation measure between time $0$ and time $t \wedge T$ of the Brownian motion $B$, so that
\begin{equation*}
F_{\gamma}^{\eps}(t) = \int_ {\R^d} e^{\gamma X_{\eps} (x) - \frac{1}{2} \gamma^2 \E[X_{\eps} (x)^2]} \nu^T_t(\d x) \,, \quad \eps \in (0, 1] \,.
\end{equation*}
Then it is sufficient to prove that $\P_B$-almost surely, for all $t \geq 0$, the measure $\nu^T_t$ has dimension two, i.e., we need to show that for $\alpha \in (0,2)$, it holds that
\begin{equation*}
\int_{\R^d \times \R^d} |x-y|^{-\alpha} \nu^T_t(\d x) \nu^T_t(\d y) < \infty \,.
\end{equation*}
Fix $t \geq 0$ and let $\nu_t$ be the occupation measure of the (unstopped) Brownian motion $B$ between time $0$ and time $t$, then it obviously holds that 
\begin{equation*}
\int_{\R^d \times \R^d}|x-y|^{-\alpha} \nu^T_t(\d x) \nu^T_t(\d y)  \leq \int_{\R^d \times \R^d} |x-y|^{-\alpha} \nu_t(\d x) \nu_t(\d y) \,.
\end{equation*}
Therefore, taking expectation, it is enough to prove that for $\alpha \in (0,2)$, it holds that
\begin{equation}
\label{eq_LemmaCons}
\E_B\left[\int_{\R^d \times \R^d} |x-y|^{-\alpha} \nu_t(\d x) \nu_t(\d y)\right] = \E_B\left[\int_0^t \int_0^t |B_r-B_s|^{-\alpha} \d r \d s\right]< \infty \,.
\end{equation}
Using Fubini's theorem, we have that 
\begin{align*}
\E_B\left[\int_{\R^d \times \R^d} |x-y|^{-\alpha} \nu_t(\d x) \nu_t(\d y)\right] & = \int_0^t \int_0^t \E_B\left[|B_r-B_s|^{-\alpha}\right] \d r \d s \\
& = \E_B\left[{|B_1|^{-\alpha}}\right]  \int_0^t\int_0^t |r-s|^{-\frac{\alpha}{2}} \d r \d s \\
& \leq 2 t \E_B\left[{|B_1|^{-\alpha}}\right] \int_0^t r^{-\frac{\alpha}{2}} \d r  \,
\end{align*}
which is clearly finite if $\alpha \in (0, 2)$. Therefore, this implies that for fixed $t \geq 0$, the measure $\nu_t^T$ has dimension two, $\P_B$-almost surely. To prove that the measure $\nu_t^T$ has dimension two for all $t \geq 0$, $\P_B$-almost surely, it is sufficient to take a sequence $(t_n)_{n \in \mathbb{N}}$ converging to infinity and noting that for $0 \leq s \leq t$, it holds that
\begin{equation*}
\int_{\R^d \times \R^d} |x-y|^{-\alpha} \nu^T_s(\d x) \nu^T_s(\d y)  \leq \int_{\R^d \times \R^d} |x-y|^{-\alpha} \nu^T_t(\d x) \nu^T_t(\d y) \,.
\end{equation*} 
The conclusion follows thanks to \cite[Theorem~1.1]{Berestycki_Elementary}.
\end{proof}
\begin{definition}
Fix $\gamma^2 < 4$ and $d \geq 2$. The $d$-dimensional LBM on $D$ associated with the field $X$ is the stochastic process $(\B^{\gamma}_t)_{t \geq 0}$ defined by 
\begin{equation}
\label{eq_LBM}
\B^{\gamma}_t := B_{F_{\gamma}^{-1}(t)} \,, \quad t \geq 0  \,,
\end{equation}
where $F_{\gamma}$ is defined as in Lemma~\ref{lm_ConstructionLBM}. 
\end{definition}
\begin{remark}
Let us emphasize that when $d \geq 3$, Lemma~\ref{lm_ConstructionLBM} yields a construction of the LBM only for $\gamma^2 < 4$, and not for $\gamma^2 < 2d$ as for the case of GMC measures. Therefore, there is a gap, for $4 \leq \gamma^2 < 2d$, where a construction of the LBM does not follow from the general theory of GMC. This is due to the well-known fact that the dimension of the occupation measure of the $d$-dimensional Brownian motion is at most two, for all $d \geq 2$ (see e.g.\ \cite[Chapter~4]{Peres}). To the best of our knowledge, if $d \geq 3$, a construction of the $d$-dimensional LBM for $4 \leq \gamma^2 < 2d$ has not be done yet. 
\end{remark} 

Before proceeding, let us introduce some notation. For $\gamma^2 <4$, we can define on the interval $[0, T]$ the measure $\mu_{\gamma}$ by setting
\begin{equation}
\label{eq_muGamma}
\mu_{\gamma}([s, t]) := F_{\gamma}(t) - F_{\gamma}(s) \,, \quad s \leq t \in [0, T] \,.
\end{equation} 
It may be worth emphasizing that an immediate consequence of Lemma~\ref{lm_ConstructionLBM} is that the measure $\mu_{\gamma}$ can be interpreted as the weak limit of the sequence of approximating measures $(\mu_{\gamma}^{\eps})_{\eps \in (0, 1]}$, where 
\begin{equation}
\label{eq_aproxMeasureMuGamma}
\mu_{\gamma}^{\eps}(\d t) := e^{\gamma X_{\eps}(B_t) - \frac{1}{2} \gamma^2 \E[X_{\eps}(B_t)^2]} \mathbbm{1}_{\{T > t\}}  \d t\,, \quad \eps \in (0, 1]\,.
\end{equation}
Similarly, we can define on the interval $[0, F_{\gamma}(T)]$ the measure $\mu_{\gamma}^{-}$ by letting
\begin{equation}
\label{eq_muGammaMinus}
\mu^{-}_{\gamma}([s, t]) := F^{-1}_{\gamma}(t) - F^{-1}_{\gamma}(s) \,, \quad s \leq t \in [0, F_{\gamma}(T)] \,,
\end{equation}  
so that for all $t \in [0, F_{\gamma}(T)]$, we can write $\B_t = B_{\mu_{\gamma}^{-}([0, t])}$.

\begin{proposition}
\label{pr_ContStrictIncr}
For $\gamma^2 < 4$, the random function $F_{\gamma}: [0, T] \to \R$ is almost surely continuous and strictly increasing on the interval $[0, T]$.
\end{proposition}
\begin{proof}
The function $F_{\gamma}$ is continuous and strictly increasing as long as the measure $\mu_{\gamma}$ has no atoms and it has full support on the interval $[0, T]$. To verify that $\mu_{\gamma}$ is almost surely atomless on $[0, T]$, one needs to use the power law behaviour of $\mu_{\gamma}$ (cf.\ Proposition~\ref{pr_PropertiesLBM} below). We refer to \cite[Remark~3.5]{Ber_LBM} for details. The proof of the fact that $\mu_{\gamma}$ has almost surely full support on $[0, T]$ can be found in \cite[Theorem~2.7]{RVG_LBM}. To be precise, in the above mentioned references the proofs are given in the planar case. However, they immediately generalize to all dimensions and hence the details are omitted.
\end{proof}

\subsection{Study of the lower singularity spectrum}
In the same spirit as in Section~\ref{sec:MRW}, our goal is to study the local regularity of the paths of the LBM $\B^{\gamma}$ defined in \eqref{eq_LBM}. More precisely, we are interested in the lower singularity spectrum $\underline{\d}_{\B^{\gamma}} : [0, \infty) \to [0, \infty)$ defined by
\begin{equation*}
\underline{\d}_{\B^{\gamma}}(\alpha) := \dim_{\mathcal{H}}(\underline{\EE}_{\B^{\gamma}}(\alpha)) \,, \quad \alpha \in [0, \infty) \,,
\end{equation*}
where $\underline{\EE}_{\B^{\gamma}}(\alpha)$ denotes the set of times at which $\B^{\gamma}$ has lower local dimension equal to $\alpha \geq 0$, i.e.\
\begin{equation*}
\underline{\EE}_{\B^{\gamma}}(\alpha) := \left\{t \in [0, F_{\gamma}(T)] \, : \, \underline{\dim}_{\B^{\gamma}}(t)= \alpha\right\} \,, 
\end{equation*}
where
\begin{equation*}
\underline{\dim}_{\B^{\gamma}}(t) :=\liminf_{r \searrow 0} \frac{\log |\B^{\gamma}_{t+r} - \B^{\gamma}_{t-r}| }{\log r} \,.
\end{equation*}
\begin{theorem}
\label{th_MainTheoremLBM}
Let $\gamma^2 < 4$ and $d \geq 2$. Then, for $\alpha \geq 0$, the lower singularity spectrum of the LBM $\B^{\gamma}$ is given by
\begin{equation*}
\underline{\d}_{\B^{\gamma}}(\alpha) = \begin{cases}
2 \alpha-2\alpha\left(\frac{2\alpha-1}{2 \alpha \gamma} + \frac{1}{4} \gamma \right)^2 \,, & \quad \text{ if } \alpha \in \left[\left(\sqrt{2}+\frac{|\gamma|}{\sqrt{2}}\right)^{-2},  \left(\sqrt{2}-\frac{|\gamma|}{\sqrt{2}}\right)^{-2}\right] \,, \\
0 \,, & \quad \text{ otherwise} \,,
\end{cases}
\end{equation*}
almost surely.
\end{theorem}

Since the LBM $\B^{\gamma}$ is a time changed Brownian motion, the lower singularity spectrum $\underline{\d}_{\B^{\gamma}}$ can be obtained through the singularity spectrum of the measure $\mu_{\gamma}^{-}$ defined in \eqref{eq_muGammaMinus}, as in the case of the MRW. However, in this setting the situation is less trivial since we do not know the singularity spectrum of $\mu_{\gamma}^{-}$ at this point. Instead of working directly with $\mu_{\gamma}^{-}$, we start by proving that the measure $\mu_{\gamma}$, defined in \eqref{eq_muGamma}, satisfies the multifractal formalism, and we find an explicit formula for its singularity spectrum. Once this is done, we can move to the proof of Theorem~\ref{th_MainTheoremLBM} which consists in finding the relations between $\underline{\d}_{\B^{\gamma}}$, $\d_{\mu_{\gamma}^{-}}$ and $\d_{\mu_{\gamma}}$. 

\subsubsection{Multifractal analysis of $\mu_{\gamma}$}
The main goal of this section is to prove that the measure $\mu_{\gamma}$ defined in  \eqref{eq_muGamma} satisfies the multifractal formalism. We will follow a strategy similar to that developed in Section~\ref{sec:main}. As we will see, the presence of the Brownian motion creates some more technical difficulties, but the underlying ideas are the same. 

We start by computing the power law spectrum of the measure $\mu_{\gamma}$. 
\begin{proposition}
\label{pr_PropertiesLBM}
Let $\gamma^2 < 4$ and $q < 4/\gamma^2$. For $r \in (0, 1]$, define the stopping time $\tau_{\sqrt{r}}:=\inf\{t \geq 0\, : \, B_t \not \in B(0, \sqrt{r}) \}$. Then there exists a finite constant $C > 0$ such that for all $r \in (0, 1]$, it holds that
\begin{equation*}
\E[\mu_{\gamma}([0, \tau_{\sqrt{r}}])^q] \leq C r^{\xi_{\mu_{\gamma}}(q)} \,,
\end{equation*}
where
\begin{equation*}
\xi_{\mu_{\gamma}}(q) := \left(1 + \frac{1}{4} \gamma^2\right)q-\frac{1}{4} \gamma^2 q^2\,, \quad q \in \R \,,
\end{equation*}
is the power law spectrum of $\mu_{\gamma}$. 
\end{proposition}
\begin{proof}
Thanks to Kahane's convexity inequality (cf.\ Lemma~\ref{lm_Kahane}), it is sufficient to prove the result only for the $d$-dimensional exactly scale invariant field $(X(x))_{x \in \R^d}$ defined in \cite{RVScale}. More precisely, for $\eps \in (0, 1]$, we let $(X_{\eps}(x))_{x \in \R^d}$ be the approximation of $X$ as defined in \cite[Lemma~3.19]{BerPow}. Then, for $\lambda \in (0, 1)$, we can assume that
\begin{equation}
\label{eq_ScaleInvarianceProp}
\left(X_{\lambda \eps}(\lambda x)\right)_{x \in B(0, 1)} \overset{d}{=} \left(X_{\eps}(x) + \Omega_{\lambda}\right)_{x \in B(0, 1)} \,, 
\end{equation}
where $\Omega_{\lambda}$ is an independent centred Gaussian random variable with variance $- \log \lambda$ (cf.\ \cite[Corollay~3.20]{BerPow}). Let $0 \leq \eps < \sqrt{r}$, then scaling time by a factor $r$ and using \eqref{eq_ScaleInvarianceProp}, we get that
\begin{align*}
\mu^{\eps}_{\gamma}([0, \tau_{\sqrt{r}}]) & = \int_{0}^{\tau_{\sqrt{r}}} e^{\gamma X_{\eps}(B_s) - \frac{1}{2} \gamma^2 \E[X_{\eps}(B_s)^2]} \d s \\
& = r \int_{0}^{\tau_{\sqrt{r}}/r} e^{\gamma X_{\eps}(B_{r s}) - \frac{1}{2} \gamma^2 \E[X_{\eps}(B_{r s})^2]} \d s \\
& \overset{d}{=} r \int_{0}^{\tilde{\tau}} e^{\gamma X_{\eps}(\sqrt{r}\tilde{B}_{s}) - \frac{1}{2} \gamma^2 \E[X_{ \eps}(\sqrt{r}\tilde{B}_{s})^2]} \d s \\
& \overset{d}{=} r e^{\gamma \Omega_{\sqrt{r}} - \frac{1}{2} \gamma^2 \E[\Omega_{\sqrt{r}}^2]} \int_{0}^{\tilde{\tau}}  e^{\gamma X_{\eps/\sqrt{r}}(\tilde{B}_{s}) - \frac{1}{2} \gamma^2 \E[X_{\eps/\sqrt{r}}(\tilde{B}_{s})^2]} \d s \,,
\end{align*}
where $\tilde{B}$ is an independent Brownian motion and $\tilde{\tau}$ is the first exit time of $\tilde{B}$ from $B(0,1)$. Now, raising to the power $q < 4/\gamma^2$ and taking expectations, one can easily verify that
\begin{equation*}
\E[\mu^{\eps}_{\gamma}([0, \tau_{\sqrt{r}}])^q] = r^{\xi(q)} \E\left[\left(\int_{0}^{\tilde{\tau}}  e^{\gamma X_{\eps/\sqrt{r}}(\tilde{B}_{s}) - \frac{1}{2} \gamma^2 \E[X_{\eps/\sqrt{r}}(\tilde{B}_{s})^2]} \d s\right)^q\right]\
\end{equation*}
Removing the approximation, i.e.\ taking the limit when $\eps \searrow 0$, gives us the desired result provided that the quantities $\mu^{\eps}_{\gamma}([0, \tau_{\sqrt{r}}])$ and $\mu_{\gamma}^{\eps/\sqrt{r}}([0, \tilde{\tau}])^q$ are uniformly integrable in $\eps \in (0, 1]$, for $q < 4/\gamma^2$. For $q \leq 0$ this fact is proved in the planar case in \cite[Proposition~2.12]{RVG_LBM}), but the proof generalizes almost verbatim in higher dimensions. For $q \in (0,4/\gamma^2)$ a proof of this fact can be found in Appendix~\ref{sec:finiteMoments}.
\end{proof}

We now state the main result that allows to prove that $\mu_{\gamma}$ satisfies the multifractal formalism. 
Let us emphasize that such a result is the analogue of Proposition~\ref{pr_RVCarrier}, but for the measure $\mu_{\gamma}$.
\begin{proposition}
\label{pr_mainLBM}
Let $\gamma^2 < 4$ and $q^2 < 4/\gamma^2$. For $\mu_{q \gamma}$-almost every $t \in [0, T]$, the measure $\mu_{\gamma}$ satisfies
\begin{equation*}
\lim_{r \searrow 0} \frac{\log \mu_{\gamma}(B(t,r))}{\log r } = 1 + \left(\frac{1}{2}-q\right) \frac{\gamma^2}{2} \,,
\end{equation*}
almost surely. 
\end{proposition}
\begin{proof}
The case in which the underlying field is a zero-boundary GFF on a bounded domain $D \subset \R^2$ has already been treated in \cite{Jackson}. The proof follows similar lines as the proof of Proposition~\ref{pr_RVCarrier} in Appendix~\ref{sec:ProofProp}. Therefore, we only highlight here the main changes that one needs to do in order for the proof to work in the general setting of a log-correlated Gaussian field $X$ on a bounded domain $D\subset \R^d$, $d\geq 2$.

The main steps of the proof in \cite{Jackson} are \cite[Lemmas~3.9 and 3.11]{Jackson}. Let us focus on \cite[Lemma~3.9]{Jackson}. The only main change that one needs to do is in \cite[Equation~3.19]{Jackson}. Indeed, in that step the author uses the domain Markov property of the GFF to bound the exponential moment of the fluctuations of the circle average approximation. However, in the case of a general log-correlated Gaussian field, we can bypass the use of such a property by using instead Lemma~\ref{lm_supremumGaussian} in order to control the fluctuations of the convolution approximation. To be more precise, let us briefly introduce some notation. Let $B$ and $\tilde{B}$ be two independent $d$-dimensional Brownian motions started from the origin, let $\F_t := \sigma(B_s \, : \, s \leq t)$ be the natural filtration associated with $B$ up to a certain fixed time $t \geq 0$, and define the stopping time
\begin{equation*}
\tilde{\tau} := \inf\{s \geq 0 \, : \, \sqrt{r} \tilde{B}_s + B_t \not \in B(0, 1/2) \} \wedge \inf\{s \geq 0 \, : \, \tilde{B}_s \not \in B(0, 1/2)\} \,,
\end{equation*}
where $r \in (0, 1]$. Then, letting $\Omega_{\sqrt{r}}^{B_t} := \inf_{x \in B(B_t, \sqrt{r})}X_{\sqrt{r}}(x) - X_{\sqrt{r}}(B_t)$ the following lower bound trivially holds on the interval $[0, \tilde{\tau}]$,
\begin{equation*}
X_{\eps \sqrt{r}}(\sqrt{r} \tilde{B}_s + B_t) \geq (X_{\eps \sqrt{r}}(\sqrt{r} \tilde{B}_s + B_t) - X_{\sqrt{r}}(\sqrt{r} \tilde{B}_s + B_t)) + \Omega_{\sqrt{r}}^{B_t} + X_{\sqrt{r}}(B_t) \,.
\end{equation*} 
Therefore, plugging this estimate into the second line of \cite[Equation~3.19]{Jackson}, using the fact that in \cite[Equation~3.20]{Jackson} the probability is conditioned on $\F_t$, and proceeding  in the same exact way as in the proof of Proposition~\ref{pr_RVCarrier1} below, we can reach the same conclusion of \cite[Lemma~3.9]{Jackson} also in our more general setting. 

Finally, the changes that one needs to do in the proof of \cite[Lemma~3.11]{Jackson} are similar and therefore not discussed.
\end{proof}

We can now state the following proposition concerning the singularity spectrum of $\mu_{\gamma}$.
\begin{proposition}
\label{pr_MuFromalsim}
If $\gamma^2 < 4$, then the measure $\mu_{\gamma}$ satisfies the multifractal formalism. More precisely, for $\alpha \geq 0$ it holds that
\begin{equation*}
\d_{\mu_{\gamma}}(\alpha) = \tau^{*}_{\mu_{\gamma}}(\alpha) = \begin{cases}
1-\left(\frac{1-\alpha}{\gamma} + \frac{1}{4} \gamma\right)^2 \,, & \quad \text{ if } \alpha \in \left[\left(1 - \frac{|\gamma|}{2}\right)^2,  \left(1 + \frac{|\gamma|}{2}\right)^2 \right] \,, \\
0 \,, & \quad \text{ otherwise} \,,
\end{cases}
\end{equation*}
almost surely.
\end{proposition}
\begin{proof}
Given Propositions \ref{pr_PropertiesLBM} and \ref{pr_mainLBM}, and defining the structure function $\phi_{\mu_{\gamma}}(q) := 1  - \xi_{\mu_{\gamma}}(q)$, for all $q \in \R$, one can proceed similarly to Section~\ref{sec:main} to prove this result. Therefore, the details are omitted.
\end{proof}

\subsubsection{Conclusion of the proof}
Now that we have found an explicit formula for the singularity spectrum of $\mu_{\gamma}$, we can study the relations between the various (lower) singularity spectra. For $\alpha \geq 0$, we define the set $\EE_{\mu_{\gamma}}(\alpha)$ as in \eqref{eq_LevelSets}, the sets $\underline{\UU}_{\mu_{\gamma}}(\alpha)$, $\overline{\BB}_{\mu_{\gamma}}(\alpha)$ as in \eqref{eq_LevelSetsExtended}, and similarly also for $\mu^{-}_{\gamma}$. Let us start with the following lemma.
\begin{lemma}
\label{lm_dimLBM}
If $\gamma^2 < 4$, then for $\alpha \geq 0$, it holds $\P_B$-almost surely that
\begin{enumerate}[label=(\roman*)]
\itemsep0em 
\item $\EE_{\mu^{-}_{\gamma}}(2 \alpha) \subset \underline{\EE}_{\B^{\gamma}}(\alpha)$;
\item $\underline{\EE}_{\B^{\gamma}}(\alpha) \subset \underline{\UU}_{\mu^{-}_{\gamma}}(2 \alpha) \cap  
\overline{\BB}_{\mu^{-}_{\gamma}}(2 \alpha)$.
\end{enumerate}
\end{lemma}
\begin{proof}
Since for all $t \in [0, T]$, it holds that $\B^{\gamma}_{t} = B_{\mu_{\gamma}^{-}((0, t])}$, one can use the same argument used in the proof of Lemma~\ref{lm_dimBM} to prove this result.
\end{proof}

Thanks to Lemma~\ref{lm_dimLBM}, the problem has been reduced to computing the Hausdorff dimensions of $\EE_{\mu^{-}_{\gamma}}(\alpha)$ and of $\underline{\UU}_{\mu^{-}_{\gamma}}(\alpha) \cap  
\overline{\BB}_{\mu^{-}_{\gamma}}(2 \alpha)$, for $\alpha \geq 0$. We start with the following trivial lemma.
\begin{lemma}
\label{lm_LBM1}
If $\gamma^2 < 4$, then for $\alpha > 0$, it holds almost surely that
\begin{enumerate}[label=(\roman*)]
\itemsep0em 
\item $\EE_{\mu^{-}_{\gamma}}(\alpha) = F_{\gamma}(\EE_{\mu_{\gamma}}(\alpha^{-1}))$;
\item $\underline{\UU}_{\mu^{-}_{\gamma}}(\alpha) \cap \overline{\BB}_{\mu^{-}_{\gamma}}(\alpha) = F_{\gamma}(\underline{\UU}_{\mu_{\gamma}}(\alpha^{-1}) \cap  
\overline{\BB}_{\mu_{\gamma}}(\alpha^{-1}))$.
\end{enumerate}
\end{lemma}
\begin{proof}
The proof follows thanks to the fact that random function $F_{\gamma}$ is almost surely continuos and strictly increasing on the interval $[0, T]$ (cf.\ Proposition~\ref{pr_ContStrictIncr}).
\end{proof}

The following lemma is concerned with the relation between $\d_{\mu_{\gamma}}$ and $\d_{\mu^{-}_{\gamma}}$.
\begin{lemma}
\label{lm_LBM2}
If $\gamma^2 < 4$, then for $\alpha > 0$, it holds almost surely that $\d_{\mu^{-}_{\gamma}}(\alpha) = \alpha \d_{\mu_{\gamma}}(\alpha^{-1})$.
\end{lemma}
\begin{proof}
Thanks to item (i) of Lemma~\ref{lm_LBM1}, we know that for $\alpha > 0$, it holds almost surely that $\EE_{\mu^{-}_{\gamma}}(\alpha) = F_{\gamma}(\EE_{\mu_{\gamma}}(\alpha^{-1}))$. Fix $\eps> 0$, and let $t \in \EE_{\mu_{\gamma}}(\alpha^{-1})$, then we define $\delta_{t} > 0$ as follows 
\begin{equation*}
\delta_{t} := \sup \left\{\delta > 0 \, : \, |F_{\gamma}(t+r) - F_{\gamma}(t-r)| \leq r^{\alpha^{-1} - \eps}, \, \forall r \in (0, \delta)\right\} \wedge 1 \,.
\end{equation*}
We define the collection of sets $(E^n_{\mu_{\gamma}}(\alpha^{-1}))_{n \in \mathbb{N}}$ by letting
\begin{equation*}
\EE^n_{\mu_{\gamma}}(\alpha^{-1}) := \left\{t \in \EE_{\mu_{\gamma}}(\alpha^{-1}) \, : \,  2^{-n} < \delta_{t} \leq  2^{-n+1} \right\} \,, \quad n \in \mathbb{N} \,.
\end{equation*}
Then $\EE_{\mu_{\gamma}}(\alpha^{-1}) = \cup_{n \in \mathbb{N}} \EE^n_{\mu_{\gamma}}(\alpha^{-1})$, and for all $t \in \EE^n_{\mu_{\gamma}}(\alpha^{-1})$, it holds that $|F_{\gamma}(t+r) - F_{\gamma}(t-r)| \leq r^{\alpha^{-1}  - \eps}$ for all $r \in (0, 2^{-n})$. Thanks to Proposition~\ref{pr_ContStrictIncr}, we know that the random function $F_{\gamma}$ is almost surely continuous and strictly increasing. Hence, using the countable stability of the Hausdorff dimension and Proposition~\ref{pr_HolderJackson}, it holds almost surely that
\begin{align*}
\d_{\mu^{-}_{\gamma}}(\alpha) = \sup_{n \in \mathbb{N}} \dim_{\H}(F_{\gamma}(\EE^n_{\mu_{\gamma}}(\alpha^{-1})))  \leq \frac{\alpha}{1 - \alpha \eps}  \sup_{n \in \mathbb{N}} \dim_{\H}(\EE^n_{\mu_{\gamma}}(\alpha^{-1})) = \frac{\alpha}{1 - \alpha \eps}  \d_{\mu_{\gamma}}(\alpha^{-1}) \,.
\end{align*}
The reverse inequality can be obtained similarly. Indeed, if we fix $\eps > 0$ and we let $t \in \EE_{\mu^{-}_{\gamma}}(\alpha)$, then we can define $\delta_{t} > 0$ as follows
\begin{equation*}
\delta_t := \sup \left\{\delta > 0 \, : \,|F^{-1}_{\gamma}(t+r) - F^{-1}_{\gamma}(t-r)| \leq r^{\alpha - \eps}, \, \forall r \in (0, \delta) \right\} \wedge 1\,.
\end{equation*}
Defining the collection of sets $(\EE^n_{\mu^{-}_{\gamma}}(\alpha))_{n \in \mathbb{N}}$ similarly as above, and proceeding in the same exact way, we obtain that it holds almost surely that
\begin{equation*}
\d_{\mu_{\gamma}}(\alpha^{-1}) \leq \frac{1}{\alpha - \eps}  \d_{\mu^{-}_{\gamma}}(\alpha) \,.
\end{equation*}
The conclusion follows by arbitrariness of $\eps > 0$.
\end{proof}

We also need the following lemma which allows to conclude the proof.
\begin{lemma}
\label{lm_LBM3}
If $\gamma^2 < 4$, then for $\alpha > 0$ it holds almost surely that $\dim_{\H}(\underline{\UU}_{\mu^{-}_{\gamma}}(\alpha) \cap \overline{\BB}_{\mu^{-}_{\gamma}}(\alpha)) \leq \alpha \tau^{*}_{\mu_{\gamma}}(\alpha^{-1})$.
\end{lemma}
\begin{proof}
Thanks to item (ii) of Lemma~\ref{lm_LBM1} and to the definition of $\tau_{\mu_{\gamma}}^{*}$, it is sufficient to prove that for $\alpha > 0$ it holds that
\begin{equation}
\label{eq_similarAppendixA}
\dim_{\mathcal{H}}(F_{\gamma}(\underline{\UU}_{\mu_{\gamma}}(\alpha^{-1}) \cap  \overline{\BB}_{\mu_{\gamma}}(\alpha^{-1}))) \leq q + \alpha \tau_{\mu_{\gamma}}(q) \,, \quad \forall q \in R \,,
\end{equation}
almost surely. We will be brief here since this result can be proved using the same procedure used in the proof of Proposition~\ref{pr_unionSet}. We will only focus in the case $q \geq 0$. Fix $\eps > 0$ and let $\overline{s}_{q, \eps} := q + (\tau_{\mu_{\gamma}}(q) + \eps)/(\alpha^{-1} - \eps)$. By definition, for every $t \in \underline{\UU}_{\mu_{\gamma}}(\alpha^{-1}) \cap \overline{\BB}_{\mu_{\gamma}}(\alpha^{-1})$, it exists a decreasing sequence $(r_{t,n})_{n \in \mathbb{N}}$ converging to $0$ such that
\begin{equation*}
|F_{\gamma}(t + r_{t,n}) - F_{\gamma}(t - r_{t,n})| = \mu_{\gamma}(B(t, r_{t, n})) \leq r_{t, n}^{\alpha^{-1} - \eps} \,, \quad \forall n \in \mathbb{N} \,.
\end{equation*} 
Fix $\delta \in (0, 1]$, then for each $t \in \underline{\UU}_{\mu_{\gamma}}(\alpha^{-1}) \cap \overline{\BB}_{\mu_{\gamma}}(\alpha^{-1})$, we choose $n_t$ such that $r_{t, n_t} \in  (0,\delta)$. For each $m \in \mathbb{N}$, we let 
\begin{equation*}
F_m := \left\{t \in\underline{\UU}_{\mu_{\gamma}}(\alpha^{-1}) \cap \overline{\BB}_{\mu_{\gamma}}(\alpha^{-1}) \, : \, 2^{-m} < r_{t, n_t} \leq 2^{-m+1}\right\} \,.
\end{equation*}
Proceeding as in the proof of Proposition~\ref{pr_unionSet}, it exists a positive integer $\sigma_1$ such that for every $m \in \mathbb{N}$, we can find $\sigma_1$ disjoint subsets $F_{m,1}, \dots, F_{m, \sigma_1}$ of $F_m$ such that each $F_{m, j}$ is at most countable, the balls $B(t, r_{t, n_t})$ with centres at $t \in F_{m, j}$ are pairwise disjoint, and the collection of sets 
\begin{equation*}
\left(\left(\left(F_{\gamma}(B(t, r_{t, n_t}))\right)_{t \in F_{m, j}}\right)_{j \in \{1, \dots, \sigma_1\}}\right)_{m \in \mathbb{N}}
\end{equation*}
is a $\delta$-cover of $F_{\gamma}(\underline{\UU}_{\mu_{\gamma}}(\alpha^{-1}) \cap \overline{\BB}_{\mu_{\gamma}}(\alpha^{-1}))$. Then, thanks to the definition of $\mu_{\gamma}$ and to the fact that $F_{\gamma}$ is almost surely strictly increasing, we have that
\begin{align*}
\mathcal{H}_{\delta}^{\overline{s}_{q, \eps}}(F_{\gamma}(\underline{\UU}_{\mu_{\gamma}}(\alpha^{-1}) \cap  \overline{\BB}_{\mu_{\gamma}}(\alpha^{-1}))) & \leq \sum_{m \in \mathbb{N}} \sum_{j = 1}^{\sigma_1} \sum_{t \in F_{m , j}} |F_{\gamma}(t + r_{t,n_t}) - F_{\gamma}(t - r_{t,n_t})|^{\overline{s}_{q, \eps}} \\
& \leq \sum_{m \in \mathbb{N}} \sum_{j = 1}^{\sigma_1} \sum_{t \in F_{m , j}} \mu_{\gamma}(B(t, r_{t, n_t}))^q r_{t, n_t}^{(\tau_{\mu}(q) + \eps)} \\
& \lesssim \sum_{m \in \mathbb{N}} \sum_{j = 1}^{\sigma_1} \sum_{t \in F_{m , j}} \mu_{\gamma}(B(t, 2^{-m+1}))^q 2^{-m(\tau_{\mu}(q) + \eps)}  \,,
\end{align*}
where the implicit constant does not depend on $m$. Hence, the conclusion follows as in the proof of Proposition~\ref{pr_unionSet}. Similarly, we can prove inequality \eqref{eq_similarAppendixA} for $q < 0$.
\end{proof}

Finally, let us see how we can put everything together in order to prove Theorem~\ref{th_MainTheoremLBM}.
\begin{proof}[Proof of Theorem~\ref{th_MainTheoremLBM}]
Let us start by observing that almost surely it holds that $\underline{\d}_{\B^{\gamma}}(0) = 0$ since $\B^{\gamma}$ is almost surely continuous. Thanks to item (i) of Lemma~\ref{lm_dimLBM} and Lemma~\ref{lm_LBM2}, for $\alpha > 0$, we have that
\begin{equation*}
\underline{\d}_{\B^{\gamma}}(\alpha) \geq \d_{\mu^{-}_{\gamma}}(2 \alpha) = 2\alpha \d_{\mu_{\gamma}}((2\alpha)^{-1})\,,
\end{equation*}
almost surely. On the other hand, item (ii) of Lemma~\ref{lm_dimLBM}, Lemma~\ref{lm_LBM3} and Proposition~\ref{pr_MuFromalsim} imply that, for all $\alpha > 0$, it holds that
\begin{equation*}
\underline{\d}_{\B^{\gamma}}(\alpha) \leq \dim_{\mathcal{H}} \left(\underline{\UU}_{\mu^{-}_{\gamma}}(2 \alpha) \cap \overline{\BB}_{\mu^{-}_{\gamma}}(2 \alpha)\right) \leq 2 \alpha \tau^{*}_{\mu_{\gamma}} ((2 \alpha)^{-1}) = 2\alpha \d_{\mu_{\gamma}}((2\alpha)^{-1}) \,,
\end{equation*}
almost surely. Since the singularity spectrum $\d_{\mu_{\gamma}}$ is known, this concludes the proof.
\end{proof}

\appendix
\addtocontents{toc}{\protect\setcounter{tocdepth}{1}}
\section{Proof of Proposition~\ref{pr_RVCarrier}}
\label{sec:ProofProp}
The proof of Proposition~\ref{pr_RVCarrier} follows similar steps of the proof in \cite[Theorem~4.1]{RV_Review}. However, our proof accommodates general log-correlated Gaussian fields. Indeed, it is not enough to prove the result for an exactly stochastically scale invariant field, and then deduce the result for all log-correlated fields, by means of Kahane's convexity inequality (cf.\ Lemma~\ref{lm_Kahane}). This is due to the fact that in the proof we need to consider moments of certain integrals of the field that involve indicator functions depending on the field itself.

Since the conclusion of Proposition~\ref{pr_RVCarrier} trivially holds for $\gamma = 0$, we may assume $\gamma \neq 0$. The proof is essentially based on the following proposition and on a standard Borel--Cantelli argument. 

\begin{proposition}
\label{pr_RVCarrier1}
Let $\gamma^2 < 2d$, $q^2 < 2d/ \gamma^2$ and define $\alpha := d+(1/2-q)\gamma^2$. For $\beta > 0$ and $E > 0$, there exist constants $C$, $b$, $C'$, $b' > 0$ such that for $r \in (0, 1]$ small enough it holds that
\begin{equation}
\label{eq_RVCarrier11}
\E\left[M_{q \gamma}\left(\left\{x \in D_r  \, : \, M_{\gamma}(B(x,r)) > E r^{\alpha-\beta}\right\}\right)\right] \leq C r^{b} \,, 
\end{equation}
and 
\begin{equation}
\E\left[M_{q \gamma}\left(\left\{x \in D_r \, : \, M_{\gamma}(B(x,r)) <  E r^{\alpha+\beta}\right\}\right)\right] \leq C' r^{b'} \label{eq_RVCarrier12} \,,
\end{equation}
where $D_r := \{x \in D \, : \, \dist(x, \partial D) < r\}$. 
\end{proposition}
\begin{proof}
The proof can be found in Subsection~\ref{subsec_Preliminary}.
\end{proof}

We are now ready to prove Proposition~\ref{pr_RVCarrier}.
\begin{proof}[Proof of Proposition~\ref{pr_RVCarrier}]
Fix $\beta > 0$ and let $\alpha = d+(1/2-q)\gamma^2$. For $n \in \mathbb{N}$, we define the following sets
\begin{equation*}
U_{\gamma}^n := \left\{x \in D_{2^{-n}} \, : \, M_{\gamma}(B(x,r)) \leq r^{\alpha-\beta} , \quad \forall r \in [0, 2^{-n})\right\} \,,
\end{equation*} 
and
\begin{equation*}
L_{\gamma}^n := \left\{x \in D_{2^{-n}} \, : \, M_{\gamma}(B(x,r)) \geq r^{\alpha+\beta} \quad \forall r \in [0, 2^{-n})\right\} \,,
\end{equation*}
where we recall that $D_r := \{x \in D \, : \, \dist(x, \partial D) < r\}$. We claim that for every $\eps \in (0, 1]$, it exists almost surely a finite $N \in \mathbb{N}$ such that 
\begin{equation*}
M_{q \gamma}\left(\left(U_{\gamma}^N\right)^c\right) < \eps \,, \quad M_{q \gamma}\left(\left(L_{\gamma}^N\right)^c\right) < \eps\,.
\end{equation*}
Thanks to Proposition~\ref{pr_RVCarrier1} with $E = 2^{-(\alpha-\beta)}$ and to Markov's inequality, we know that there exist constants $C > 0$ and $b > 0$ such that for $r \in (0, 1]$ small enough, it holds that
\begin{align*}
& \P\left(M_{q\gamma}\left(\left\{x \in D_r \, : \, M_{\gamma}(B(x, r)) > 2^{-(\alpha - \beta)} r^{\alpha-\beta}\right\}\right) \geq r^{\frac{b}{2}}\right) \\
& \qquad\qquad \leq r^{-\frac{b}{2}} \E\left[M_{q \gamma}\left(\left\{x \in D_r \, : \, M_{\gamma}(B(x,r)) > 2^{-(\alpha - \beta)} r^{\alpha-\beta}\right\}\right)\right] \\
& \qquad\qquad \leq C r^{\frac{b}{2}} \,.
\end{align*}
For $n \in \mathbb{N}$, taking $r = 2^{-n}$, thanks to the Borel--Cantelli lemma, we get that the events
\begin{equation*}
\left(\left\{M_{q\gamma}\left(\left\{x \in D_{2^{-n}} \, : \, M_{\gamma}(B(x, 2^{-n})) > (2^{-n-1})^{\alpha - \beta}\right\}\right) \geq (2^{-n})^{\frac{b}{2}}\right\}\right)_{n \in \mathbb{N}}
\end{equation*}
occur only finitely often almost surely. Therefore, for all $\eps \in (0, 1]$, it exists almost surely a finite $N \in \mathbb{N}$ such that 
\begin{equation*}
M_{q\gamma}\left(\bigcup_{n \geq N} \left\{x \in D_{2^{-n}} \, : \, M_{\gamma}(B(x, 2^{-n})) > (2^{-n-1})^{\alpha - \beta}\right\}\right) \leq \sum_{n \geq N} (2^{-n})^{\frac{b}{2}} < \eps \,.
\end{equation*}
Let $x \in (U_{\gamma}^N)^c$, then it exists $r \in [0, 2^{-N})$ such that $M_{\gamma}(B(x, r)) > r^{\alpha - \beta}$. Since it exists $n \geq N$ such that $2^{-n-1} \leq r < 2^{-n}$, we get
\begin{equation*}
M_{\gamma}(B(x, 2^{-n})) \geq M_{\gamma}(B(x, r)) > r^{\alpha - \beta} \geq (2^{-n-1})^{\alpha-\beta} \,,
\end{equation*}
which implies that
\begin{equation*}
M_{q \gamma}\left(\left(U_{\gamma}^N\right)^c\right) \leq M_{q\gamma}\left(\bigcup_{n \geq N} \left\{x \in D_{2^{-n}} \, : \, M_{\gamma}(B(x, 2^{-n})) > (2^{-n-1})^{\alpha - \beta}\right\}\right) < \eps \,.
\end{equation*}
A similar argument can be used to prove that for every $\eps \in (0, 1]$, it exists almost surely a finite $N \in \mathbb{N}$ such that $M_{q \gamma}((L_{\gamma}^N)^c) < \eps$. Therefore, introducing the sets
\begin{align*}
& U_{\gamma} := \bigcup_{n \in \mathbb{N}} U_{\gamma}^n = \left\{x \in D \, : \, \liminf_{r \searrow 0} \frac{\log M_{\gamma}(B(x,r))}{\log r} \geq \alpha-\beta \right\} \,, \\
& L_{\gamma} := \bigcup_{n \in \mathbb{N}} L_{\gamma}^n = \left\{x \in D \, : \, \limsup_{r \searrow 0} \frac{\log M_{\gamma}(B(x,r))}{\log r} \leq \alpha+\beta \right\} \,,
\end{align*}
since $\eps > 0$ is arbitrary and $U_{\gamma}^N$, $L_{\gamma}^N$ are increasing sets, we obtain that
\begin{equation*}
M_{q \gamma}\left(\left(L_\gamma \cap U_{\gamma}\right)^{c}\right) = 0 \,.
\end{equation*}
The conclusion then follows from the arbitrariness of $\beta > 0$.
\end{proof}

\subsection{Proof of Proposition~\ref{pr_RVCarrier1}}
\label{subsec_Preliminary}%
The objective of this subsection is to prove Proposition~\ref{pr_RVCarrier1}. For simplicity, but without loss of generality, we assume that $D = (-1,1)^d$. The proof is based on Lemma~\ref{lm_supremumGaussian} and on the following two lemmas.

\begin{lemma}
\label{lm_ProofCarrierRV}
For $\gamma^2 < 2d$ and $q^2 < 2d/\gamma^2$, there exist $\bar{a} \in (0, 1]$ and $C > 0$ such that 
\begin{equation*}
\sup_{\eps \in (0, 1]} \E \left[\left(\int_{(-1,1)^d}\frac{e^{\gamma X_{\eps}(x)-\frac{1}{2} \gamma^2 \E[X_{\eps}(x)^2]}}{(|x|+\eps)^{q \gamma^2}} \d x\right)^{\bar{a}}\right] \leq C
\end{equation*}
\end{lemma}
\begin{proof}
The proof can be found in Subsection~\ref{subsec_Preliminary1}.
\end{proof}
\begin{lemma}
\label{lm_ProofCarrierRVNegative}
For $\gamma^2 < 2d$ and $q^2 < 2d/ \gamma^2$, there exists $C > 0$ such that 
\begin{equation*}
\sup_{\eps \in (0, 1]} \E \left[\left(\int_{[-1/2,1/2]^d}\frac{e^{\gamma X_{\eps}(x)-\frac{1}{2} \gamma^2 \E[X_{\eps}(x)^2]}}{(|x|+\eps)^{q \gamma^2}} \d x\right)^{-1}\right] \leq C \,.
\end{equation*}
\end{lemma}
\begin{proof}
The proof can be found in Subsection~\ref{subsec_Preliminary1}.
\end{proof}

\begin{proof}[Proof of Proposition~\ref{pr_RVCarrier1}]
We split the proof into two parts. We start by proving \eqref{eq_RVCarrier11} and then we move to the proof of \eqref{eq_RVCarrier12}. For simplicity, but without loss of generality, we consider the case $E = 1$. 

\vspace*{2mm}
{\it Proof of \eqref{eq_RVCarrier11}.} Let $r \in (0, 1]$ and $x \in D_r$, where we recall that we are assuming that $D = (-1,1)^d$. Fix $\eps$, $\eps' > 0$ so that $\eps' < r \eps$. Then by Girsanov's theorem (cf.\ Lemma~\ref{lm_Girsanov}), we have that
\begin{align}
\label{eq_startLemma1}
& \E\left[\mathbbm{1}_{\left\{\int_{B(x,r)} e^{\gamma X_{r\eps}(u) - \frac{1}{2} \gamma^2 \E[X_{r \eps}(u)^2]} \d u > r^{\alpha-\beta}\right\}} e^{q \gamma X_{\eps'}(x)-\frac{1}{2} q^2 \gamma^2 \E[X_{\eps'}(x)^2]}\right] \nonumber \\
& \qquad \qquad = \P\left(\int_{B(x,r)} e^{\gamma X_{r\eps}(u) +q \gamma^2 \E[X_{r \eps}(u)X_{\eps'}(x)] - \frac{1}{2} \gamma^2 \E[X_{r \eps}(u)^2]} \d u > r^{\alpha-\beta}\right) \nonumber \\
& \qquad\qquad \leq \P\left(e^{ \gamma^2 K |q|} \int_{B(x,r)} \frac{e^{\gamma X_{r\eps}(u) - \frac{1}{2} \gamma^2 \E[X_{r \eps}(u)^2]}}{(|u-x|+r\eps)^{q \gamma^2}} \d u > r^{\alpha-\beta}\right) \,,
\end{align}
where the last inequality is justified by the fact that, thanks to property \ref{pp_P3}, for $x$, $y \in (-1,1)^d$, it exists $K>0$ such that 
\begin{equation*}
-\log\left(|x-y| + r \eps\right) - K \leq \E[X_{r \eps}(x) X_{\eps'}(y)] \leq -\log\left(|x-y| + r \eps\right) + K \,.
\end{equation*}
Now, let $\Omega_r^x := \sup_{u \in B(x,r)} X_r(u)-X_{r}(x)$, then by changing variables, we see that the integral in \eqref{eq_startLemma1} can be upper bounded as follows
\begin{align}
\label{eq_likeScaleInv1}
& \int_{B(x,r)} \frac{e^{\gamma X_{r\eps}(u) - \frac{1}{2} \gamma^2 \E[X_{r \eps}(u)^2]}}{(|u-x|+r\eps)^{q \gamma^2}} \d u \nonumber \\
& \qquad\qquad \leq e^{|\gamma| (\Omega_r^x + X_r(x))}  \int_{B(x,r)} \frac{e^{\gamma (X_{r \eps}(u) - X_{r}(u)) - \frac{1}{2} \gamma^2 \E[X_{r \eps}(u)^2]}}{(|u-x|+r\eps)^{q \gamma^2}} \d u \nonumber \\
& \qquad\qquad =  r^{d-q \gamma^2} e^{|\gamma| (\Omega_r^x + X_r(x))} \int_{B(0,1)} \frac{e^{\gamma (X_{r \eps}(r u + x) - X_{r}(r u + x)) - \frac{1}{2} \gamma^2 \E[X_{r \eps}(r u +x)^2]}}{(|u|+\eps)^{q \gamma^2}} \d u \nonumber \\
& \qquad\qquad \leq r^{d+(1/2-q)\gamma^2} e^{\frac{1}{2}\gamma^2 K + |\gamma| (\Omega_r^x + X_r(x))} \int_{B(0,1)} \frac{e^{\gamma (X_{r \eps}(r u + x) - X_{r}(r u + x)) + \frac{1}{2} \gamma^2 \log \eps }}{(|u|+\eps)^{q \gamma^2}} \d u \nonumber \\
& \qquad\qquad \leq r^{\alpha} e^{\gamma^2 K c + |\gamma| (\Omega_r^x + X_r(x))} \int_{B(0,1)} \frac{e^{\gamma X'_{\eps}(u) - \frac{1}{2} \gamma^2 \E[X'_{\eps}(u)^2]}}{(|u|+\eps)^{q \gamma^2}} \d u \,,
\end{align}
for some log-correlated field $X'$ in $B(0,1)$, and some finite constant $c > 0$ depending only on $K$. Substituting the last expression of \eqref{eq_likeScaleInv1} back into \eqref{eq_startLemma1}, we obtain that
\begin{align}
\label{eq_IntegratingAux}
& \E\left[\mathbbm{1}_{\left\{\int_{B(x,r)} e^{\gamma X_{r\eps}(u) - \frac{1}{2} \gamma^2 \E[X_{r \eps}(u)^2]} \d u > r^{\alpha-\beta}\right\}} e^{q \gamma X_{\eps'}(x)-\frac{1}{2} q^2 \gamma^2 \E[X_{\eps'}(x)^2]}\right] \\
& \quad \leq \P\left(e^{\gamma^2 K (c + |q|) + |\gamma| (\Omega_r^x + X_r(x))} \int_{B(0,1)} \frac{e^{\gamma X'_{\eps}(u) - \frac{1}{2} \gamma^2 \E[X'_{\eps}(u)^2]}}{(|u|+\eps)^{q \gamma^2}} \d u > r^{-\beta}\right) \nonumber \\
& \quad \leq \P\left(e^{ \gamma^2 K (c + |q|) + |\gamma| \Omega_r^x}> r^{-\frac{\beta}{3}}\right) + \P\left(e^{|\gamma| X_r(x)}> r^{-\frac{\beta}{3}}\right) + \P\left(\int_{B(0,1)} \frac{e^{\gamma X'_{\eps}(u) - \frac{1}{2} \gamma^2 \E[X'_{\eps}(u)^2]}}{(|u|+\eps)^{q \gamma^2}} \d u > r^{-\frac{\beta}{3}}\right) \nonumber \,.
\end{align}
Let us notice that all the three terms decay polynomially in $r$, for $r \in (0, 1]$ small enough, uniformly in $x \in D_r$ and $\eps$, $\eps' \in (0, 1]$. Indeed, for the first term this follows from Lemma~\ref{lm_supremumGaussian} and Markov's inequality. For the second term, this follows from the fact that $X_r(x)$ is a centred Gaussian random variable with variance $\E[X_{r}(x)^2] = -\log r + \mathcal{O}(1)$ (cf.\ \cite[Lemma~12.9]{Peres}). Finally, for the third term this follows from Lemma~\ref{lm_ProofCarrierRV} and Markov's inequality again. Therefore, for $r \in (0,1]$ small enough, integrating over $x \in D_r$ the expression in \eqref{eq_IntegratingAux} and putting everything together, we can see that there exist finite constants $C$, $b > 0$ such that for all $\eps$, $\eps' \in (0, 1]$, it holds that
\begin{equation*}
\E\left[\int_{D_r} \mathbbm{1}_{\left\{\int_{B(x,r)} e^{\gamma Y_{r\eps}(u) - \frac{1}{2} \gamma^2 \E[Y_{r \eps}(u)^2]} \d u > r^{\alpha-\beta}\right\}} e^{q \gamma Y_{\eps'}(x)-\frac{1}{2} q^2 \gamma^2 \E[Y_{\eps'}(x)^2]} \d x \right] \leq C r^{b} \,.
\end{equation*}
Now, thanks to Proposition~\ref{pr_consGausChaos}, we know that the sequence of measures $(M_{q\gamma}^{\eps'})_{\eps' \in (0,1]}$ converges to $M_{q \gamma}$ in $L^1(\P)$ in the topology of weak convergence of measures. Therefore, thanks to Portmanteau lemma and Fatou's lemma, we get that
\begin{align*}
& \E\left[\int_{D_r} \mathbbm{1}_{\left\{\int_{B(x,r)} e^{\gamma Y_{r\eps}(u) - \frac{1}{2} \gamma^2 \E[Y_{r \eps}(u)^2]} \d u > r^{\alpha-\beta}\right\}} M_{q \gamma}(\d x)\right] \\
& \qquad\qquad \leq \liminf_{\eps' \searrow 0}\E\left[\int_{D_r} \mathbbm{1}_{\left\{\int_{B(x,r)} e^{\gamma Y_{r\eps}(u) - \frac{1}{2} \gamma^2 \E[Y_{r \eps}(u)^2]} \d u > r^{\alpha-\beta}\right\}} e^{q \gamma Y_{\eps'}(x)-\frac{1}{2} q^2 \gamma^2 \E[Y_{\eps'}(x)^2]} \d x \right] \,.
\end{align*}
Therefore, since $C$ and $b$ are independent from $\eps' \in(0, 1]$, we have that
\begin{equation}
\label{eq_lastPropo1RV}
\E\left[\int_{D_r} \mathbbm{1}_{\left\{\int_{B(x,r)} e^{\gamma Y_{r\eps}(u) - \frac{1}{2} \gamma^2 \E[Y_{r \eps}(u)^2]} \d u > r^{\alpha-\beta}\right\}} M_{q \gamma}(\d x)\right] \leq C r^{b} \,.
\end{equation}
Finally, since $C$ and $b$ are independent also from $\eps \in(0, 1]$, using \eqref{eq_lastPropo1RV} and Fatou's lemma, we see that
\begin{align*}
& \E\left[\int_{D_r} \mathbbm{1}_{\left\{M_{\gamma}(B(x, r)) > r^{\alpha-\beta}\right\}} M_{q \gamma}(\d x)\right] \\
& \qquad\qquad \leq \liminf_{\eps \searrow 0} \E\left[\int_{D_r} \mathbbm{1}_{\left\{\int_{B(x,r)} e^{\gamma Y_{r\eps}(u) - \frac{1}{2} \gamma^2 \E[Y_{r \eps}(u)^2]} \d u > r^{\alpha-\beta}\right\}} M_{q \gamma}(\d x)\right] \\
& \qquad\qquad \leq C r^{b} \,.
\end{align*}

\vspace*{2mm}
{\it Proof of \eqref{eq_RVCarrier12}.} We shall be brief here since the argument is similar to that used in the proof of \eqref{eq_RVCarrier11}. Let $r \in (0, 1]$ and $x \in D_r$, where we recall that we are assuming that $D = (-1,1)^d$. Fix $\eps$, $\eps' \in (0, 1]$ so that $\eps' < r \eps$. Then by Girsanov's theorem (cf.\ Lemma~\ref{lm_Girsanov}) and property \ref{pp_P3}, we have
\begin{align}
\label{eq_startLemma2}
& \E\left[\mathbbm{1}_{\left\{\int_{B(x,r)} e^{\gamma X_{r\eps}(u) - \frac{1}{2} \gamma^2 \E[X_{r \eps}(u)^2]} \d u < r^{\alpha+\beta}\right\}} e^{q \gamma X_{\eps'}(x)-\frac{1}{2} q^2 \gamma^2 \E[X_{\eps'}(x)^2]}\right] \nonumber \\
& \qquad\qquad \leq \P\left(e^{-\gamma^2 K |q|} \int_{B(x,r)} \frac{e^{\gamma X_{r\eps}(u) - \frac{1}{2} \gamma^2 \E[X_{r \eps}(u)^2]}}{(|u-x|+r\eps)^{q \gamma^2}} \d u < r^{\alpha+\beta}\right) \,.
\end{align}
Now, let $\Omega_r^x := \inf_{u \in B(x,r)} X_r(u)-X_{r}(x)$, then by proceeding in the same way as in the proof of \eqref{eq_RVCarrier11}, we can see that the right-hand side of \eqref{eq_startLemma2} can be lower bounded as follows
\begin{align}
\label{eq_likeScaleInv2}
& \int_{B(x,r)} \frac{e^{\gamma X_{r\eps}(u) - \frac{1}{2} \gamma^2 \E[X_{r \eps}(u)^2]}}{(|u-x|+r\eps)^{q \gamma^2}} \d u  \nonumber \\
& \qquad \qquad \geq r^{\alpha} e^{-\gamma^2 K c + |\gamma|(\Omega_r^x + X_r(x))} \int_{B(0,1)} \frac{e^{\gamma X'_{\eps}(u) - \frac{1}{2} \gamma^2 \E[X'_{\eps}(u)^2]}}{(|u|+\eps)^{q \gamma^2}} \d u \,,
\end{align}
for some log-correlated field $X'$ in $B(0,1)$, and some finite constant $c > 0$ depending only on $K$. Substituting the last expression of \eqref{eq_likeScaleInv2} back into \eqref{eq_startLemma2}, we obtain that
\begin{align*}
& \E\left[\mathbbm{1}_{\left\{\int_{B(x,r)} e^{\gamma X_{r\eps}(u) - \frac{1}{2} \gamma^2 \E[X_{r \eps}(u)^2]} \d u < r^{\alpha+\beta}\right\}} e^{q \gamma X_{\eps'}(x)-\frac{1}{2} q^2 \gamma^2 \E[X_{\eps'}(x)^2]}\right] \\
& \quad \leq \P\left(e^{-\gamma^2 K (c+|q|) + |\gamma| \Omega_r^x} < r^{\frac{\beta}{3}}\right) + \P\left(e^{|\gamma| X_r(x)} < r^{\frac{\beta}{3}}\right) + \P\left(\int_{B(0,1)} \frac{e^{\gamma X'_{\eps}(u) - \frac{1}{2} \gamma^2 \E[X'_{\eps}(u)^2]}}{(|u|+\eps)^{q \gamma^2}} \d u < r^{\frac{\beta}{3}}\right) \,.
\end{align*}
Thanks to Lemmas~\ref{lm_supremumGaussian} and \ref{lm_ProofCarrierRVNegative}, the conclusion then follows as in the proof of \eqref{eq_RVCarrier11}.
\end{proof}

\subsection{Proofs of Lemmas~\ref{lm_ProofCarrierRV} and \ref{lm_ProofCarrierRVNegative}}
\label{subsec_Preliminary1}
Let us finish by proving Lemmas~\ref{lm_ProofCarrierRV} and \ref{lm_ProofCarrierRVNegative}. We recall once again that, for simplicity, we are assuming that $D = (-1,1)^d$. 

\begin{proof}[Proof of Lemma~\ref{lm_ProofCarrierRV}]
For $\eps \in (0, 1]$ and $a \in (0, 1]$, we define the quantity
\begin{equation*}
U^{a}_{\eps} := \E \left[\left(\int_{(-1,1)^d}\frac{e^{\gamma X_{\eps}(x)-\frac{1}{2} \gamma^2 \E[X_{\eps}(x)^2]}}{(|x|+\eps)^{q \gamma^2}} \d x\right)^{a}\right] \,,
\end{equation*}
and we need to prove that it exists $\bar{a} \in (0, 1]$ such that $U^{\bar{a}}_{\eps}$ can be uniformly bounded in $\eps \in (0, 1]$. We split the proof into two parts. In the first part we focus on $q \in (-\sqrt{2d}/|\gamma|, 0]$, while in the second part we focus on $q \in (0, \sqrt{2d}/|\gamma|)$. 

\vspace*{2mm}
{\it Case $q \in (-\sqrt{2d}/|\gamma|, 0]$.} Since $|x|+\eps \leq \sqrt{d}+1$ for all $x \in (-1,1)^d$ and $\eps \in (0, 1]$, we can write 
\begin{equation*}
U^{a}_{\eps} \leq (\sqrt{d}+1)^{-a q \gamma^2} \E \left[\left(\int_{(-1,1)^d}e^{\gamma X_{\eps}(x)-\frac{1}{2} \gamma^2 \E[X_{\eps}(x)^2]} \d x\right)^{a}\right]  \leq C  (\sqrt{d}+1)^{-a q \gamma^2}  \,,
\end{equation*}
where the existence of the constant $C >0$, independent of $\eps \in (0, 1]$, follows from Jensen's inequality. 

\vspace*{2mm}
{\it Case $q \in (0, \sqrt{2d}/|\gamma|)$.} If $\eps \in (1/2,1]$, then we have that 
\begin{equation*}
U_{\eps}^{a} \leq 2^{a q \gamma^2} \E \left[\left(\int_{(-1,1)^d}e^{\gamma X_{\eps}(x)-\frac{1}{2} \gamma^2 \E[X_{\eps}(x)^2]} \d x\right)^{a}\right] \leq C 2^{a q \gamma^2} \,.
\end{equation*}   
where the existence of the constant $C >0$, independent of $\eps \in (1/2, 1]$, follows from Jensen's inequality. Let us now assume that $\eps \in (0, 1/2]$. Then it exists a unique $n \in \mathbb{N}$ such that $2^{-n-1} < \eps \leq 2^{-n}$. For all $k \in \{0, 1, \dots, n\}$, we let $Q_k := (-2^{-k}, 2^{-k})^d$. Then by sub-additivity of the function $x \mapsto x^{a}$ for $a \in (0, 1]$, we have that
\begin{align}
\label{eq_boundUeps}
U^{a}_{\eps} & \leq \E \left[\left(\int_{Q_n}\frac{e^{\gamma X_{\eps}(x)-\frac{1}{2} \gamma^2 \E[X_{\eps}(x)^2]}}{(|x|+\eps)^{q \gamma^2}} \d x\right)^{a}\right] + \sum_{k = 0}^{n-1}\E \left[\left(\int_{Q_{k} \setminus Q_{k+1}}\frac{e^{\gamma X_{\eps}(x)-\frac{1}{2} \gamma^2 \E[X_{\eps}(x)^2]}}{(|x|+\eps)^{q \gamma^2}} \d x\right)^{a}\right] \nonumber \\
& \leq 2^{n a(q \gamma^2 - d)}\E \left[\left(\int_{(-1,1)^d}\frac{e^{\gamma X_{\eps}(2^{-n}x)-\frac{1}{2} \gamma^2 \E[X_{\eps}(2^{-n}x)^2]}}{(|x|+2^{n}\eps)^{q \gamma^2}} \d x\right)^{a}\right] \nonumber \\
& \qquad\qquad + \sum_{k =0}^{n-1} 2^{(k+1) a q \gamma^2}\E \left[\left(\int_{Q_{k} \setminus Q_{k+1}} e^{\gamma X_{\eps}(x)-\frac{1}{2} \gamma^2 \E[X_{\eps}(x)^2]} \d x\right)^{a}\right] \nonumber \\
& \leq 2^{n a(q \gamma^2 - d)}\E \left[\left(\int_{(-1,1)^d} e^{\gamma X_{\eps}(2^{-n} x)-\frac{1}{2} \gamma^2 \E[X_{\eps}(2^{-n} x)^2]} \d x\right)^{a}\right] \nonumber \\
& \qquad\qquad + 2^{a d} \sum_{k = 0}^{n-1} 2^{(k+1) a (q \gamma^2-d)}\E \left[\left(\int_{(-1,1)^d} e^{\gamma X_{\eps}(2^{-k}x)-\frac{1}{2} \gamma^2 \E[X_{\eps}(2^{-k} x)^2]} \d x\right)^{a}\right] \,.
\end{align} 
Now let $k \in \{0, 1, \dots, n\}$, then we want to estimate the following quantity
\begin{equation*}
\E \left[\left(\int_{(-1,1)^d} e^{\gamma X_{\eps}(2^{-k} x)-\frac{1}{2} \gamma^2 \E[X_{\eps}(2^{-k} x)^2]} \d x\right)^{a}\right] \,.
\end{equation*}
Thanks to property \ref{pp_P3}, for $x$, $y \in (-1,1)^d$, it exists $K>0$ such that
\begin{equation*}
\E[X_{\eps}(2^{-k} x) X_{\eps}(2^{-k} y)] \geq  -\log\left(|x -y|+2^{k}\eps\right) + \log\left(2^{k}\right) - K \,,
\end{equation*}
and therefore, if we let $Z$ and $\Omega_{2^{-k}}$ be two real normal random variables independent from everything else with variance $2K$ and $\log 2^k$, respectively, then
\begin{equation*}
\E[(X_{\eps}(2^{-k} x)+Z ) (X_{\eps}(2^{-k} y)+Z)] \geq \E[\left(X_{2^k \eps}(x)+\Omega_{2^{-k}}\right)\left(X_{2^k \eps}(y)+\Omega_{2^{-k}}\right)] \,.
\end{equation*}
Since the function $x \mapsto x^{a}$ for $a \in (0, 1]$ is concave, thanks to Kahane's convexity inequality (cf.\ Lemma~\ref{lm_Kahane}), we have that
\begin{align*}
& \E \left[\left(\int_{(-1,1)^d} e^{\gamma X_{\eps}(2^{-k} x)-\frac{1}{2} \gamma^2 \E[X_{\eps}(2^{-k} x)^2]} \d x\right)^{a}\right] \\
& \qquad\qquad  = e^{\gamma^2 a K (1-a)} \E \left[\left(\int_{(-1,1)^d} e^{\gamma \left( X_{\eps}(2^{-k} x)+Z \right)-\frac{1}{2} \gamma^2 \E\left[\left(X_{\eps}(2^{-k} x) + Z \right)^2\right]} \d x\right)^{a}\right] \\
& \qquad\qquad  \leq e^{\gamma^2 a K (1-a)} \E \left[\left(\int_{(-1,1)^d} e^{\gamma \left(X_{2^{k} \eps}(x) + \Omega_{2^{-k}}\right) -\frac{1}{2} \gamma^2 \E\left[\left(X_{2^{k} \eps}(x) +\Omega_{2^{-k}}\right)^2\right]} \d x\right)^{a}\right] \\
& \qquad\qquad  = e^{\gamma^2 a K (1-a)} \E \left[e^{\gamma a \Omega_{2^{-k}} - \frac{1}{2} a \gamma^2 \E\left[\Omega_{2^{-k}}^2\right]}\right] \E \left[\left(\int_{(-1,1)^d} e^{\gamma X_{2^{k} \eps}(x) - \frac{1}{2} \gamma^2 \E\left[X_{2^{k} \eps}(x)^2\right]}\d x\right)^{a}\right] \\
& \qquad\qquad  \leq C e^{\gamma^2 a K (1-a)} 2^{\frac{1}{2} \gamma^2 k a (a-1)} \,,
\end{align*}
where the existence of the constant $C >0$, independent of $\eps \in (0, 1/2]$, follows from Jensen's inequality. Hence, going back to \eqref{eq_boundUeps}, letting 
\begin{equation}
s_{a} :=\left(d + \frac{1}{2}\gamma^2 - q\gamma^2\right)a - \frac{1}{2}\gamma^2 a^2 \,,
\end{equation}
we obtain that
\begin{equation*}
U^{a}_{\eps} \lesssim \sum_{k =0}^n 2^{-k s_{a}} \,,
\end{equation*}
where the implicit constant does not depend on $n$. It can be easily verified that for each $\gamma^2 < 2d$ and $q \in (0, \sqrt{2d}/|\gamma|)$ it exists $\bar{a} \in (0, 1]$ such that $s_{\bar{a}} > 0$. In particular, this implies that $U^{\bar{a}}_{\eps}$ is uniformly bounded in $\eps \in [0, 1/2)$, which proves the result. 
\end{proof}

\begin{proof}[Proof of Lemma~\ref{lm_ProofCarrierRVNegative}]
For every $\eps \in (0, 1]$, we define 
\begin{equation*}
U^{-1}_{\eps} := \E \left[\left(\int_{[-1/2,1/2]^d}\frac{e^{\gamma X_{\eps}(x)-\frac{1}{2} \gamma^2 \E[X_{\eps}(x)^2]}}{(|x|+\eps)^{q \gamma^2}} \d x\right)^{-1}\right] \,.
\end{equation*}
We need to prove that $U^{-1}_{\eps}$ can be uniformly bounded in $\eps \in (0, 1]$. We split the proof into two cases. 

\vspace*{2mm}
{\it Case $q \in [0, \sqrt{2d}/|\gamma|)$.} In this case, since $|x|+\eps \leq \sqrt{d/2}+1$ for all $x \in [-1/2,1/2]^d$ and $\eps \in (0, 1]$, we can write 
\begin{equation*}
\sup_{\eps \in (0,1]} U_{\eps}^{-1} \leq (\sqrt{d/2}+1)^{q \gamma^2} \sup_{\eps \in (0,1]}\E \left[\left(\int_{[-1/2,1/2]^d}e^{\gamma X_{\eps}(x)-\frac{1}{2} \gamma^2 \E[X_{\eps}(x)^2]} \d x\right)^{-1}\right] \leq C (\sqrt{d/2}+1)^{q \gamma^2}\,,
\end{equation*}
where the existence of the constant $C >0$, independent of $\eps \in (0, 1]$, follows from Proposition~\ref{pr_UnifBound}.

\vspace*{2mm}
{\it Case $q \in (-\sqrt{2d}/|\gamma|,0)$.} In this case, we can simply proceed as follows
\begin{equation*}
U_{\eps}^{-1} \leq \E \left[\left(\int_{B(1/4, 1/8)}\frac{e^{\gamma X_{\eps}(x)-\frac{1}{2} \gamma^2 \E[X_{\eps}(x)^2]}}{(|x|+\eps)^{q \gamma^2}} \d x\right)^{-1}\right] \,,
\end{equation*}
and since $|x|+\eps \geq 1/8$ for all $x \in B(1/4, 1/8)$ and $\eps \in (0, 1]$, we can write 
\begin{align*}
\sup_{\eps \in (0, 1]} U_{\eps}^{-1} \leq (1/8)^{-q \gamma^2} \sup_{\eps \in (0, 1]} \E \left[\left(\int_{B(1/4, 1/8)} e^{\gamma X_{\eps}(x)-\frac{1}{2} \gamma^2 \E[X_{\eps}(x)^2]} \d x\right)^{-1}\right] \leq C (1/8)^{-q \gamma^2} \,,
\end{align*}
and also in this case the existence of the constant $C >0 $, independent of $\eps \in (0, 1]$, follows from Proposition~\ref{pr_UnifBound}.
\end{proof}

\section{Finiteness of positive moments of $\mu_{\gamma}$}
\label{sec:finiteMoments}
In this section, we adopt the same notation used in Section~\ref{sec:LBM}. The main goal is to prove that the moments of order $q \in (0, 4/\gamma^2)$ of the measure $\mu_{\gamma}$, defined in \eqref{eq_muGamma}, are uniformly bounded. The techniques used in the proof of this result are inspired from \cite[Section~3]{BerPow}. Moreover, let us mention that this result has been proved in the planar case in \cite[Subsection~3.2]{Ber_LBM}, but the proof does not trivially generalize to all dimensions. Before stating and proving the result, it may be worth recalling that the measure $\mu_{\gamma}$ can be interpreted as the weak limit of the sequence of approximated measures $(\mu_{\gamma}^{\eps})_{\eps \in (0, 1]}$ defined in \eqref{eq_aproxMeasureMuGamma}.
\begin{proposition}
Let $\gamma^2 < 4$ and $q \in (0, 4/\gamma^2)$. Then it holds that 
\begin{equation*}
\sup_{\eps \in (0, 1]}\E\left[\left(\int_0^{T} e^{\gamma X_{\eps}(B_s) - \frac{1}{2} \E[X_{\eps}(B_s)^2]}] \d s\right)^{q} \right]< \infty \,,
\end{equation*}
where we recall that $T$ denotes the first exit time of $B$ from the bounded domain $D$.
\end{proposition}
\begin{proof} 
As in the proof of Proposition~\ref{pr_PropertiesLBM}, thanks to Kahane's convexity inequality (cf.\ Lemma~\ref{lm_Kahane}), we can assume $X$ to be the $d$-dimensional exactly scale invariant field. Moreover, without loss of generality, we can consider $q>1$.  In order to keep the notation not too cumbersome, we will give the proof only for $q \in (1, 2 \wedge 4/\gamma^2)$. Notice that this is not restrictive if $\gamma^2 \in [2, 4)$. The case in which $\gamma^2 < 2$ and $q \in [2, 4/\gamma^2)$ is discussed at the end.

For simplicity, we consider the case $D = (0,1)^d$. Let us fix $\eps \in (0,1]$, then letting $a = q-1 \in (0,1)$, using Fubini's theorem, Girsanov's theorem (cf.\ Lemma~\ref{lm_Girsanov}), and the analogue of property~\ref{pp_P3} for the exactly scale invariant field $X_{\eps}$, we obtain that
\begin{align}
\label{eq_FinMom0}
& \E\left[\left(\int_0^{T}e^{\gamma X_{\eps}(B_s)-\frac{1}{2}\gamma^2\E[X_{\eps}(B_s)^2]}  \d s \right)^{q}\right] \nonumber \\
& \qquad\qquad = \E_B\left[\int_0^{T} \E_X\left[e^{\gamma X_{\eps}(B_s)-\frac{1}{2}\gamma^2\E[X_{\eps}(B_s)^2]} \left(\int_0^{T}e^{\gamma X_{\eps}(B_r)-\frac{1}{2}\gamma^2\E[X_{\eps}(B_r)^2]}  \d r\right)^{a}\right] \d s \right]  \nonumber \\
& \qquad\qquad = \E_B\left[\int_0^{T} \E_X\left[\left(\int_0^{T}e^{\gamma X_{\eps}(B_r) + \gamma^2 \E_X[X_{\eps}(B_s)X_{\eps}(B_r)] -\frac{1}{2}\gamma^2\E[X_{\eps}(B_r)^2]}  \d r\right)^{a}\right] \d s \right] \nonumber \\
& \qquad\qquad \lesssim \E_B \left[\int_0^{T} \E_X\left[\left(\int_0^{T} \frac{e^{\gamma X_{\eps}(B_r) - \frac{1}{2}\gamma^2\E[X_{\eps}(B_r)^2]} }{(|B_r-B_s| + \eps)^{\gamma^2}} \d r\right)^{a}\right]\d s\right] \,.
\end{align}
If $\eps \in (1/2,1]$, then thanks to the concavity of the function $x \mapsto x^{a}$ for $a \in (0, 1)$, using Jensen's inequality, we have that 
\begin{equation*}
\E_X\left[\left(\int_0^{T} \frac{e^{\gamma X_{\eps}(B_r) - \frac{1}{2}\gamma^2\E[X_{\eps}(B_r)^2]} }{(|B_r-B_s| + \eps)^{\gamma^2}} \d r\right)^{a}\right] \leq 2^{a \gamma^2}\E_X\left[\int_0^{T} e^{\gamma X_{\eps}(B_r) - \frac{1}{2}\gamma^2\E[X_{\eps}(B_r)^2]}\d r\right]^{a} = 2^{a\gamma^2} T^{a} \,.
\end{equation*}
Therefore, substituting this estimate into \eqref{eq_FinMom0}, we get that 
\begin{equation*}
\E\left[\left(\int_0^{T}e^{\gamma X_{\eps}(B_s)-\frac{1}{2}\gamma^2\E[X_{\eps}(B_s)^2]}  \d s \right)^{q}\right] \leq 2^{a\gamma^2} \E_B[T^{1+a}] < \infty \,,
\end{equation*}
where the last inequality follows from the fact that $T$ has exponentially decaying tail. Let us now assume that $\eps \in (0, 1/2]$. Then it exists a unique $n \in \mathbb{N}$ such that $2^{-n-1} < \eps \leq 2^{-n}$. For $x \in \R^d$ and $k \in \{0, 1, \dots, n\}$, we let $Q_k(x)$ be the open box of side length $2^{-k+1}$ centred at $x$, i.e.\ $Q_k(x) = (-2^{-k}, 2^{-k}) + x$. Then by sub-additivity of the function $x \mapsto x^{a}$ for $a \in (0, 1)$, we have that
\begin{align}
\label{eq_FinMom10}
&\E_X\left[\left(\int_0^{T} \frac{e^{\gamma X_{\eps}(B_r) - \frac{1}{2}\gamma^2\E[X_{\eps}(B_r)^2]} }{(|B_r-B_s| + \eps)^{\gamma^2}} \d r\right)^{a}\right] \nonumber \\
& \qquad\qquad\leq \E_X\left[\left(\int_0^{T} \frac{e^{\gamma X_{\eps}(B_r) - \frac{1}{2}\gamma^2\E[X_{\eps}(B_r)^2]} }{(|B_r-B_s| + \eps)^{\gamma^2}} \mathbbm{1}_{\{B_r \in Q_n(B_s)\}} \d r\right)^{a}\right] \nonumber \\
& \qquad\qquad\qquad\qquad + \sum_{k = 1}^{n}  \E_X\left[\left(\int_0^{T} \frac{e^{\gamma X_{\eps}(B_r) - \frac{1}{2}\gamma^2\E[X_{\eps}(B_r)^2]} }{(|B_r-B_s| + \eps)^{\gamma^2}} \mathbbm{1}_{\{B_r \in Q_{k-1}(B_s) \setminus Q_{k}(B_s)\}} \d r\right)^{a}\right] \nonumber \\
& \qquad\qquad \leq 2^{n a \gamma^2 } \E_X\left[\left(\int_0^{T} \frac{e^{\gamma X_{\eps}(B_r) - \frac{1}{2}\gamma^2\E[X_{\eps}(B_r^2]} }{(2^n|B_r-B_s| + 2^n\eps)^{\gamma^2}} \mathbbm{1}_{\{B_r \in Q_n(B_s)\}} \d r\right)^{a}\right] \nonumber \\
& \qquad\qquad\qquad\qquad + \sum_{k = 1}^{n} 2^{k a \gamma^2 }  \E_X\left[\left(\int_0^{T} e^{\gamma X_{\eps}(B_r) - \frac{1}{2}\gamma^2\E[X_{\eps}(B_r)^2]} \mathbbm{1}_{\{B_r \in Q_{k-1}(B_s)\}} \d r\right)^{a}\right] \nonumber \\
& \qquad\qquad \leq 2^{n a \gamma^2 } \E_X\left[\left(\int_0^{T} e^{\gamma X_{\eps}(B_r) - \frac{1}{2}\gamma^2\E[X_{\eps}(B_r)^2]} \mathbbm{1}_{\{B_r \in Q_{n}(B_s)\}} \d r\right)^{a}\right] \nonumber \\
& \qquad\qquad\qquad\qquad + \sum_{k = 1}^{n} 2^{k a \gamma^2 }  \E_X\left[\left(\int_0^{T} e^{\gamma X_{\eps}(B_r) - \frac{1}{2}\gamma^2\E[X_{\eps}(B_r)^2]} \mathbbm{1}_{\{B_r \in Q_{k-1}(B_s)\}} \d r\right)^{a}\right] \,.
\end{align}
Now let $k \in \{0, 1, \dots, n\}$, then we want to estimate the following quantity
\begin{equation*}
\E_X\left[\left(\int_0^{T} e^{\gamma X_{\eps}(B_r) - \frac{1}{2}\gamma^2\E[X_{\eps}(B_r)^2]} \mathbbm{1}_{\{B_r \in Q_{k}(B_s)\}} \d r\right)^{a}\right] \,.
\end{equation*}
Using the translation and scaling invariance of the field $X_{\eps}$, Jensen's inequality, and Fubini's theorem, we can proceed as follows
\begin{align}
\label{eq_FinMom2}
& \E_X\left[\left(\int_0^{T} e^{\gamma X_{\eps}(B_r) - \frac{1}{2}\gamma^2\E[X_{\eps}(B_r)^2]} \mathbbm{1}_{\{B_r \in Q_{k}(B_s)\}} \d r\right)^{a}\right]  \nonumber \\
& \qquad = \E_X\left[\left(\int_0^{T} e^{\gamma X_{\eps}(2^{-k}(2^k(B_r-B_s))+B_s) - \frac{1}{2}\gamma^2\E[X_{\eps}(2^{-k}(2^k(B_r-B_s))+B_s)^2]} \mathbbm{1}_{\{2^k(B_r-B_s) \in Q_{0}(0)\}}  \d r\right)^{a}\right]  \nonumber \\
& \qquad = \E_X\left[e^{\gamma a\Omega_{2^{-k}} - \frac{1}{2} \gamma^2 a \E[\Omega_{2^{-k}}^2]}\left(\int_0^{T} e^{\gamma X_{2^k\eps}(2^k(B_r-B_s)) - \frac{1}{2}\gamma^2\E[X_{2^k\eps}(2^k(B_r-B_s))^2]} \mathbbm{1}_{\{2^k(B_r-B_s) \in Q_0(0)\}} \d r\right)^{a}\right]  \nonumber \\
& \qquad \leq 2^{k(\frac{1}{2} \gamma^2 a^2-\frac{1}{2} \gamma^2 a)} \E_X\left[\int_0^{T} e^{\gamma X_{2^k\eps}(2^k(B_r-B_s)) - \frac{1}{2}\gamma^2\E[X_{2^k \eps}(2^k(B_r-B_s))^2]}  \mathbbm{1}_{\{2^k(B_r-B_s) \in Q_0(0)\}}  \d r\right]^{a}  \nonumber \\
& \qquad = 2^{k(\frac{1}{2} \gamma^2 a^2-\frac{1}{2} \gamma^2 a)} \left(\int_0^T \mathbbm{1}_{\{2^k(B_r-B_s) \in Q_0(0)\}} \d r\right)^a\,,
\end{align}
where we recall that $\Omega_{2^{-k}}$ denotes an independent centred Gaussian random variable with variance $\log(2^k)$. Therefore, plugging estimate \eqref{eq_FinMom2} into \eqref{eq_FinMom10}, we obtain that
\begin{equation*}
\E_X\left[\left(\int_0^{T} \frac{e^{\gamma X_{\eps}(B_r) - \frac{1}{2}\gamma^2\E[X_{\eps}(B_r)^2]} }{(|B_r-B_s| + \eps)^{\gamma^2}} \d r\right)^{a}\right] \lesssim \sum_{k = 0}^n 2^{k(\frac{1}{2} \gamma^2 a^2+\frac{1}{2} \gamma^2 a)} \left(\int_0^T \mathbbm{1}_{\{2^k(B_r-B_s) \in Q_0(0)\}} \d r\right)^a\,.
\end{equation*}
Hence, integrating the previous expression over $s \in [0, T]$ and then taking expectation with respect to $B$, recalling inequality \eqref{eq_FinMom0}, we have that 
\begin{align}
\label{eq_FinMomFin1}
& \E\left[\left(\int_0^{T}e^{\gamma X_{\eps}(B_s)-\frac{1}{2}\gamma^2\E[X_{\eps}(B_s)^2]}  \d s \right)^{q}\right] \nonumber \\
& \qquad\qquad \lesssim \sum_{k = 0}^n 2^{k(\frac{1}{2} \gamma^2 a^2+\frac{1}{2} \gamma^2 a)} \E_B\left[\int_0^T \left(\int_0^T \mathbbm{1}_{\{2^k(B_r-B_s) \in Q_0(0)\}} \d r\right)^a \d s \right] \,.
\end{align}
Let us now estimate the expectation with respect to $B$ on the right-hand side of the above expression. For all $k\in\{0, 1, \dots, n\}$, applying Jensen's inequality two times, we obtain that 
\begin{align*}
\E_B\left[\int_0^T \left(\int_0^T \mathbbm{1}_{\{2^k(B_r-B_s) \in Q_0(0)\}} \d r\right)^a \d s \right] & \leq \E_B\left[T^{\frac{{1-a}}{a}}  \int_0^T\int_0^T \mathbbm{1}_{\{2^k(B_r-B_s) \in Q_0(0)\}} \d r\d s\right]^{a} \,.
\end{align*}
In particular, letting $\alpha \in (0, 2)$, using Markov's inequality, Cauchy--Schwarz's inequality, and proceeding similarly to the proof of Lemma~\ref{lm_ConstructionLBM}, we have that
\begin{align*}
& \E_B\left[T^{\frac{{1-a}}{a}}  \int_0^T\int_0^T \mathbbm{1}_{\{2^k(B_r-B_s) \in Q_0(0)\}} \d r\d s\right] \\
& \qquad\qquad = \sum_{n \in \mathbb{N}} \E_B\left[ \mathbbm{1}_{\{n-1 \leq T < n\}} T^{\frac{{1-a}}{a}}  \int_0^T\int_0^T \mathbbm{1}_{\{2^k(B_r-B_s) \in Q_0(0)\}} \d r\d s \right]\\
& \qquad\qquad \leq \sum_{n \in \mathbb{N}} n^{\frac{{1-a}}{a}}  \E_B\left[\mathbbm{1}_{\{n-1 \leq T < n\}} \int_0^n\int_0^n \mathbbm{1}_{\{2^k(B_r-B_s) \in Q_0(0)\}} \d r\d s \right] \\
& \qquad\qquad \leq \sum_{n \in \mathbb{N}} n^{\frac{{1-a}}{a}}  \P_B(T \geq n-1)^{\frac{1}{2}} \int_0^n\int_0^n \P_B(|B_r-B_s| < 2^{-k+1})^{\frac{1}{2}}\d r\d s \\
& \qquad\qquad \leq 2^{-\alpha (k-1)} \E_B[|B_1|^{-2\alpha}]^{\frac{1}{2}} \sum_{n \in \mathbb{N}} 2 n^{\frac{{1-a}}{a}+1}  \P_B(T \geq n-1)^{\frac{1}{2}} \int_0^n r^{-\frac{\alpha}{2}} \d r \,,
\end{align*}
and the latter series is obviously finite since we chose $\alpha \in (0,2)$ and $T$ has exponentially decaying tail. Therefore, going back to \eqref{eq_FinMomFin1}, we have that 
\begin{equation*}
\E\left[\left(\int_0^{T}e^{\gamma X_{\eps}(B_s)-\frac{1}{2}\gamma^2\E[X_{\eps}(B_s)^2]}  \d s \right)^{q}\right] \lesssim \sum_{k = 0}^n 2^{ka(\frac{1}{2} \gamma^2 a+\frac{1}{2} \gamma^2-\alpha)} \,,
\end{equation*}
where the implicit constant does not depend on $n$. To conclude it is sufficient to notice that for $a < 4/\gamma^2 -1$, there exists $\alpha \in (0,2)$ such that $\gamma^2 a/2 + \gamma^2/2-\alpha < 0$. 

As a final remark, we emphasize that, if $\gamma^2 < 2$, we only proved the existence
of moments of order $q \in (0,2)$. However, the strategy of the proof could be adapted to
treat the general case $q \in (0, 4/\gamma^2)$. In this case, one has to choose $n \in \mathbb{N}$ such that $n\leq q <n+1$. Then applying Girsanov's theorem $n$ times in \eqref{eq_FinMom0}, one obtain an expression similar to the last line of \eqref{eq_FinMom0} except that we get an integral with $n$ singularities. Then, we can reproduce the argument up to modifications that are obvious but notationally heavy.
\end{proof}

\section{Gaussian toolbox}
\label{sec:GaussianTool}
We collect here some well-known results on Gaussian fields. In all the subsequent lemmas, we assume $D$ to be a bounded domain of $\R^d$, $d \geq 1$. 
\begin{lemma}[{\cite[Theorem~2.1]{LN_GMC}}]
\label{lm_Girsanov}
Consider an almost surely continuous centred Gaussian field $(X(x))_{x \in D}$ and a Gaussian random variable $Z$ which belongs to the $L^2$ closure of the vector space spanned by $(X(x))_{x \in D}$. Let $F: C(D) \to \R$ be a bounded continuos functional. Then the following equality holds  
\begin{equation*}
\E\left[e^{Z-\frac{\E[Z^2]}{2}} F\left(X(\cdot)\right)\right] = \E\left[F\left(X(\cdot)+\E[X(\cdot)Z]\right)\right] \,.
\end{equation*}
\end{lemma}
Let us observe that Lemma~\ref{lm_Girsanov} is equivalent to the fact that under the probability measure $e^{Z-\E[Z^2]/2} \d \P$ the field $(X(x))_{x \in D}$ has the same law of the shifted field $(X(x)+\E[X(x)Z])_{x \in D}$ under $\P$, which is the usual Girsanov's theorem.

We now present a fundamental tool in the study of GMC measures, which is Kahane's convexity inequality. Essentially, this is an inequality that allows to compare GMC measures associated with two slightly different fields.
\begin{lemma}[{\cite[Theorem~2.2]{LN_GMC}}]
\label{lm_Kahane}
Consider two almost surely continuous centred Gaussian fields $(X(x))_{x \in D}$ and $(Y(x))_{x \in D}$ such that
\begin{equation*}
\E[X(x) X(y)] \leq \E[Y(x)Y(y)] \,, \quad \forall x, y \in D \,.
\end{equation*} 
Let $f : (0, \infty) \to \mathbb{R}$ a convex function with at most polynomial growth at $0$ and $\infty$, and let $\nu$ be a Radon measure on $D$ as in \cite[Equation~1.3]{Berestycki_Elementary}. Then it holds that
\begin{equation*}
\E\left[f\left(\int_D e^{X(x)-\frac{1}{2}\E[X(x)^2]} \nu(\d x) \right)\right] \leq \E\left[f\left(\int_D e^{Y(x)-\frac{1}{2}\E[Y(x)^2]} \nu(\d x) \right)\right] \,.
\end{equation*}
\end{lemma}

We have the following standard concentration inequality for Gaussian fields which is known as Borell-TIS inequality.
\begin{lemma}[{\cite[Theorem~2.1.1]{Adler}}]
\label{lm_Borell}
Consider an almost surely continuous centred Gaussian field $(X(x))_{x \in D}$. Then it holds that
\begin{equation*}
\P\left(\left|\sup_{x \in D} X(x) - \E\left[\sup_{x \in D} X(x)\right]\right| > t \right) \leq 2 e^{-\frac{t^2}{2 \sigma_D^2}} \,,
\end{equation*}
for all $t \geq 0$, where $\sigma_D^2 := \sup_{x \in D} \E[X(x)^2]$.
\end{lemma}

We finish this appendix with Dudley's entropy bound. 
\begin{lemma}[{\cite[Theorem~1.3.3]{Adler}}]
\label{lm_Dudley}
Consider an almost surely continuous centred Gaussian field $(X(x))_{x \in D}$. Consider the pseudo-metric on $D \times D$ defined as follows
\begin{equation*}
d_X(x, y) := \sqrt{\E[|X(x)-X(y)|^2]}, \quad  (x, y) \in D \times D,
\end{equation*}
and let $\diam_X(D) := \sup_{(x, y) \in D \times D} d_X(x, y)$. Let $N(\eps, D, d_X)$ be the number of balls of radius $\eps$ with respect to $d_X$ needed to cover $D$. Then there exists a universal constant $C$ such that 
\begin{equation*}
\E\left[\sup_{x \in D} X(x)\right] < C \int_{0}^{\diam_X(D)/2} \sqrt{\log(N(\eps, D, d_X))} \d \eps \,.
\end{equation*}
\end{lemma}

\small
\def\cprime{$'$}


\end{document}